\documentclass[12pt,sumlimits,intlimits]{amsart}
\usepackage{url,cite,hyperref,amsmath,amsthm,amssymb,mathtools}
\usepackage[margin=1.25in]{geometry}
\usepackage{xcolor}
\usepackage{comment}

\newtheorem{theorem}{Theorem}[section]
\newtheorem{proposition}[theorem]{Proposition}
\newtheorem{corollary}[theorem]{Corollary}
\newtheorem{lemma}[theorem]{Lemma}
\theoremstyle{definition}
\newtheorem{definition}[theorem]{Definition}
\newtheorem{remark}[theorem]{Remark}

\newtheorem{question}[theorem]{Question}
\newcommand{\ww}{\mathcal{W}}
\newcommand{\la}{\langle}
\newcommand{\ra}{\rangle}

\newcommand{\V}{\Vert}

\newcommand{\Z}{\mathbb{Z}}
\newcommand{\F}{\mathbb{F}}
\newcommand{\C}{\mathbb{C}}

\newcommand{\ve}{\varepsilon}

\numberwithin{equation}{section}

\title{Free groups and quasidiagonality}
\author{Caleb Eckhardt}
\address{Department of Mathematics, Miami University, Oxford, OH, 45056}
\email{eckharc@miamioh.edu}
\thanks{The author was partially supported by a grant from the Simons Foundation.}
\begin{document}
\maketitle
\begin{abstract} We use free groups to settle a couple questions about the values of the Pimsner-Popa-Voiculescu modulus of quasidiagonality for a set of operators $\Omega$, denoted by $\textup{qd}(\Omega)$.  Along the way we deduce information about the operator space structure of finite dimensional subspaces of $\C[\F_d]\subseteq C^*_{\ell^p}(\F_d)$ where 
$C^*_{\ell^p}(\F_d)$ is the so-called $\ell^p$-completion of $\C[\F_d].$
 Roughly speaking, we use free groups and $\textup{qd}(\Omega)$ to put a quantitative face on the two known qualitative obstructions to quasidiagonality; absence of an amenable trace or the presence of a proper isometry. The modulus of quasidiagonality for a proper isometry is equal to 1.  We show that $\textup{qd}(\{\lambda_a,\lambda_b\})\in [1/2,\sqrt{3}/2]$ where $a$ and $b$ are free group generators and $\lambda$ is the left regular representation.
 In another direction,  we use certain $\ell^p$ representations of free groups constructed by Pytlik and Szwarc and a recent result of Ruan and Wiersma to show that $\textup{qd}(\Omega)$ may be positive, yet arbitrarily close to zero when $\Omega$ is a set of unitaries.
\end{abstract}
\section{Introduction} In their investigations of C*-algebra extensions \cite{Pimsner79}, Pimsner, Popa and Voiculescu introduced the modulus of quasidiagonality for a set of operators. For a set of bounded operators $\Omega$ on a Hilbert space, the modulus of quasidiagonality, written as $\textup{qd}(\Omega)$, measures how badly the set $\Omega$ violates quasidiagonality.  In particular, $\textup{qd}(\Omega)=0$ if and only if $\Omega$ is a quasidiaognal set of operators (see Definition \ref{def:modqd}).   In \cite{Carrion13} the author, along with Carri\`{o}n and Dadarlat, established a connection between the modulus of quasidiagonality and the Turing number of a non-amenable group.  This connection (see specifically  \cite[Theorem 2.4 and 2.5]{Carrion13}) hinted at the possibility of putting a quantitative face on the qualitative obstructions of quasidiagonality. The first part of this paper was motivated by this possibility.
 
There are two known qualitative obstructions to quasidiagonality. Proper isometries provide the first obstruction. If $S\in B(H)$ is a proper isometry, then it is fairly easy to see that $S$ is not a quasidiagonal operator.  It is nearly as easy to see that $\textup{qd}(\{ S \})=1$ (see Proposition \ref{prop:propisom});  the maximum value of $\textup{qd}$ for a set of contractions. Amenability provides the second obstruction.  Rosenberg showed in \cite{Hadwin87} that for any non-amenable group $\Gamma$, the reduced C*-algebra $C^*_r(\Gamma)$ is not quasidiagonal \footnote{N. Brown clarified and generalized this obstruction with his introduction of amenable traces in \cite{Brown06}}. In particular if $\Gamma'\subseteq \Gamma$ generates $\Gamma$, then for $\Omega=\{ \lambda_g:g\in \Gamma' \}\subseteq B(\ell^2\Gamma)$ we have $\textup{qd}(\Omega)>0.$ We also have the trivial upper bound of $\textup{qd}(\Omega)\leq1.$ 

This produces a natural question: Is it possible that $\textup{qd}(\Omega)<1$, in other words, can $\Omega$ violate quasidiagonality in a quantitatively different way than a proper isometry?

We show the answer is yes when $\Gamma=\F_2$; the rank 2 free group. In particular, we prove in Theorem \ref{theorem:mainthm}  that for generators $a,b\in \F_2$ we have $\textup{qd}(\{ \lambda_a,\lambda_b \})\leq \sqrt{3}/2.$  The lower bound $\textup{qd}(\{ \lambda_a,\lambda_b \})\geq 1/2$ was shown in \cite[Theorem 2.4]{Carrion13}.  In fact, \cite[Theorem 2.4]{Carrion13} showed that for any non-amenable group $\Gamma$ there is a lower bound on $\textup{qd}(\{ \lambda_g:g\in \Gamma' \})$ where $\Gamma'$ consists of the coefficients of a paradoxical decomposition for $\Gamma$ and the lower bound is in terms of the number of pieces in a minimal paradoxical decomposition (also known as the Turing number of $\Gamma$). This raised the next question answered in this paper: For a set of unitaries $\Omega$ is it possible for $\textup{qd}(\Omega)$ to be positive yet as close to 0 as one wishes? 

The previous paragraph points out that it would be difficult to find an answer from the left regular representation of a discrete group.
We therefore turned to so-called $\ell^p$-representations of free groups for the answer.
In \cite{Brown13}, Brown and Guentner made some nice contributions to the study of ``$D$-representations" for discrete groups that inspired our consideration of $\ell^p$-representations to answer the modulus of quasidiagonality question of the previous paragraph. 

We recall some definitions from \cite{Brown13}. 
\begin{definition} \label{def:lprep}
Let $\Gamma$ be a discrete group and for $2\leq p\leq \infty$ consider $\ell^p(\Gamma).$  A unitary representation $(\pi,H)$ of $\Gamma$ is an $\ell^p$-\textbf{representation} if there is a dense  subspace $K\subseteq H$ such that for all $\xi,\eta\in K$, the matrix coefficient $t\mapsto \la \pi(t)\xi,\eta  \ra$ is in $\ell^p(\Gamma).$ One then defines a norm on $\C[\Gamma]$ as 
\begin{equation*}
\V x \V_{C^*_{\ell^p}(\Gamma)}=\sup\{ \V \pi(x) \V: \pi \textrm{ is an }\ell^p\textrm{-representation of }\Gamma  \}
\end{equation*}
and defines $C^*_{\ell^p}(\Gamma)$ to be the completion of $\C[\Gamma]$ with respect to this norm.  
\end{definition}
It is clear that the identity map on $\C[\Gamma]$ extends to a *-homomorphism from $C^*_{\ell^q}(\Gamma)$ onto $C^*_{\ell^p}(\Gamma)$ whenever $p\leq q.$
It was shown in \cite{Brown13} that $C^*_{\ell^2}(\Gamma)\cong C^*_r(\Gamma)$, the reduced group C*-algebra, and 
$C^*_{\ell^\infty}(\Gamma)\cong C^*(\Gamma)$, the full group C*-algebra, for any discrete group.  

Let $d\geq2$ and consider the rank $d$ free group  $\F_d$. The C*-algebras $C^*_{\ell^p}(\F_d)$ are mysterious. It was shown in \cite[Proposition 4.4]{Brown13} that there is \emph{some} $p\in (2,\infty)$ such that $C^*_{\ell^p}(\F_d)$ is not isomorphic to either $C^*_r(\F_d)$ or $C^*(\F_d).$ Okayasu showed in \cite[Corollary 3.7]{Okayasu14} that the natural *-homomorphism from $C^*_{\ell^q}(\F_d)$ onto $C^*_{\ell^p}(\F_d)$ is not injective when $p<q$ which provides some evidence that perhaps $C^*_{\ell^q}(\F_d)\not\cong C^*_{\ell^p}(\F_d).$ These results birthed the crude question, does $C^*_{\ell^p}(\F_d)$ behave more like $C^*_r(\F_d)$ or more like $C^*(\F_d)$?  This question was wide open until Ruan and Wiersma made inroads into the local structure of these C*-algebras in \cite{Ruan16}.

Interestingly, for $2<p<\infty$ the C*-algebras $C^*_{\ell^p}(\F_d)$ sometimes behave like $C^*(\F_d)$ and sometimes like $C^*_r(\F_d).$ It was shown in \cite[Theorem 4.3]{Ruan16} that if $2<p<\infty$, then $C^*_{\ell^p}(\F_d)$ is not locally reflexive--a fact shared with $C^*(\F_d)$, but not with the locally reflexive C*-algebra $C^*_r(\F_d).$ On the other hand, they also showed in \cite[Corollary 5.3]{Ruan16} that $C^*_{\ell^p}(\F_d)$ is not quasidiagonal--a fact shared with $C_r^*(\F_d)$ but not with the quasidiagonal C*-algebra $C^*(\F_d).$  In the second part of the paper we use the modulus of quasidiagonality to quantify this result.

Let $\Omega_p \subseteq C^*_{\ell_p}(\F_d)$ be the image of a generating set for $\F_d.$  By Ruan and Wiersma's result we have $\textup{qd}(\Omega_p)>0$ for all $p<\infty.$  Using the $\ell^p$-representations of $\F_d$ beautifully constructed by Pytlik and Szwarc in \cite{Pytlik86}, we show in Theorem \ref{thm:arbclose0} that 
\begin{equation*}
\lim_{p\rightarrow\infty}\textup{qd}(\Omega_p)=0.   
\end{equation*}
We finish the paper with an application to the theory of operator spaces. Let $X\subseteq \C[\F_d]$ be a subspace. Let $X_p\subseteq C^*_{\ell^p}(\F_d)$ be its canonical image.
In Corollary \ref{cor:BMdistancearb}  we show (again relying critically on \cite{Pytlik86}) that for any finite dimensional subspace $X\subset \C[\F_d]$ the completely bounded Banach-Mazur distance $d_{cb}(X_p,X_q)$ (see Definition \ref{def:dcb}) varies continuously with the parameters $p,q\in [2,\infty].$
%%%%%%%%%%%%%%%%%%%%%%%%%%%%%%%%%%%%%%%%%%%%%%%%%%%%%%%%%%%%%%%%%%%%%%%%%%%%%%%%%
%%%%%%%%%%%%%%%%%%%%%%%%%%%%%%%%%%%%%%%%%%%%%%%%%%%%%%%%%%%%%%%%%%%%%%%%%%%%%%%%%
%%%%%%%%%%%%%%%%%%%%%%%%%%%%%%%%%%%%%%%%%%%%%%%%%%%%%%%%%%%%%%%%%%%%%%%%%%%%%%%%%
\section{Modulus of quasidiagonality for two free group generators} \label{sec:mqdfree}
We first recall the definition of $\textup{qd}(\Omega)$ from \cite[Section 5]{Pimsner79}
\begin{definition} \label{def:modqd}
 Let $H$ be a Hilbert space and $P(H)$ the set of finite rank projections in $B(H).$ Then $P(H)$ is partially ordered by the usual ordering on self-adjoint elements of $B(H).$ For operators $T,S\in B(H)$ let $[T,S]=TS-ST.$ Let  $\Omega\subseteq B(H)$ be a finite subset.   Then the \textbf{modulus of quasidiagonality} for $\Omega$, written $\textup{qd}(\Omega)$ is defined as
\begin{equation}
\textup{qd}(\Omega)=\liminf_{P\in P(H)}(\max_{T\in \Omega}\V [T,P] \V).
\end{equation}
A set of operators $\Omega$ is called \textbf{quasidiagonal} if $\textup{qd}(\Omega)=0.$ Notice that $\textup{qd}(\Omega)=0$ if and only if there is an increasing net of finite rank projections $P_i$ that tend to the identity strongly with $\lim [T,P_i]=0$ for all $T\in \Omega.$ A C*-algebra $A$ is \textbf{quasidiagonal} if there is a faithful representation $\pi$ of $A$ such that $\Omega$ is quasidiagonal for every finite subset of $\pi(A).$
\end{definition}
There are two known obstructions to quasidiagonality; the presence of a proper isometry \footnote{Technically we should say the presence of a proper isometry in $M_n(C^*(\Omega))$ for some $n.$} or the presence of a non-amenable trace.  We refer the reader to \cite[Chapter 7]{Brown08} for a full discussion on these obstructions. This section estimates $\textup{qd}(\Omega)$ for these situations. Let us start with the easy case.
\subsection{Obstruction 1: Proper isometry}
The following easy argument is certainly well-known but we did not find it written down anywhere.
\begin{proposition} \label{prop:propisom} Let $H$ be a Hilbert space and $S\in B(H)$ a proper isometry, i.e. $SS^*\neq 1.$  Then $\textup{qd}(\{ S \})=1.$
\end{proposition}
\begin{proof} Let $\eta\in \textup{ker}(S^*)$ be norm 1. Let $P$ be any finite rank projection with $P\eta=\eta.$ 
Then  $\textup{rank}(SP)=\textup{rank}(P)$, while $\textup{rank}(PS)=\textup{rank}(PSS^*P)<\textup{rank}(P)$ because $PSS^*P\eta=0.$ Hence there is a non-zero vector $\xi\in \textup{ker}(PS)\cap \textup{range}(P)$, whence $\V SP-PS \V\geq1.$ Since $P$ was an arbitrary finite rank projection dominating the rank one projection onto $\C\eta$, we have $\textup{qd}(\{ S \})\geq1.$  Combining this with the trivial estimate $\textup{qd}(\{ S \})\leq \V S \V=1$ we obtain the conclusion.  

\end{proof}
\subsection{Obstruction 2: Non-amenability}    
Fix two generators $a,b\in \F_2$ and consider $\{\lambda_a,\lambda_b\}\subseteq C^*_r(\F_2)\subseteq B(\ell^2\F_2),$ where $\lambda$ is the left regular representation.  It was shown in \cite[Theorem 2.4]{Carrion13} that $\textup{qd}(\{\lambda_a,\lambda_b\})\geq 1/2.$  In this section we prove a non-trivial upper bound.  The idea is quite simple, but unfortunately the calculations are somewhat technical.  Let us explain the idea before we begin the calculations. 

The operator $\lambda_a\in B(\ell^2\F_2)$ is unitarily equivalent to infinitely many copies of the bilateral shift on $\ell^2(\Z).$  Indeed just partition $\F_2$ into right cosets relative to the subgroup generated by $a$, then $\lambda_a$ acting on the left  leaves each coset invariant and acts as the shift on each coset. Without much effort one can build an increasing sequence of finite rank projections that tend strongly to the identity and approximately commute with a bilateral shift. Therefore with slightly more effort we can do the same for $\lambda_a$ since it is a bunch of copies of the bilateral shift. In the language of Definition \ref{def:modqd}; $\textup{qd}(\{ \lambda_a \})=0.$

We also see that $\lambda_a$ and $\lambda_b$ are unitarily equivalent to each other via the unitary, call it $U$, induced by the automorphism of $\F_2$ that interchanges $a$ and $b.$  One can diagonalize $U$ as $\left(\begin{array}{rr} 1 & 0\\ 0& -1\end{array}\right).$  If $P$ is a projection that approximately commutes with $\lambda_a$, then $UPU^*$ is a projection that approximately commutes with $\lambda_b.$  Our very simple plan is to then twist a projection $P$ by a unitary ``half-way" between $U$ and the identity--we use the operator $V=\left(\begin{array}{rr} 1 & 0\\ 0& i\end{array}\right).$ We then estimate the norm of the commutators $[\lambda_a,VPV^*]$ and $[\lambda_b,VPV^*]$ and hope that they are halfway between $0$ and $1.$
A few ``back of the envelope" calculations provided good evidence that this simple plan could work.  In this section we leave the envelope and provide all of the details. For some readers the above two paragraphs are probably proof enough of the estimate.  For those readers who would at least like to see the ``back of the envelope" calculations we will point those out at the beginning of the proof of Theorem \ref{theorem:mainthm}.  

\subsection{Setup}  
\begin{definition} Let $|\cdot|$ be the length function on $\F_2$ subject to the generating set $\{a,b,a^{-1},b^{-1}\}.$ Let $e\in \F_2$ be the identity element. 
For an integer $R\geq1$, let 
\begin{equation*}
B_R=\{ x\in \F_2: |x|\leq R\}.
\end{equation*}
Notice that the cardinality of $B_R$ is $1+4\cdot 3^{R-1}$. For each $x\in \F_2$, let $W_x\subseteq \F_2$ be those reduced words that begin with $x.$ For multiple $x_1,...,x_n\in \F_2$ we let $W_{x_1,...,x_n}=\cup_{i=1}^n W_{x_i}.$ For notational convenience we will write $W_e$ for the singleton set $\{ e \}.$
\end{definition}
 \begin{definition} \label{defn:Ueigen}
 Let $\alpha:\F_2\rightarrow \F_2$ be the automorphism that interchanges $a$ and $b.$  Let $\{ \delta_x: x\in \F_2 \}\subseteq \ell^2\F_2$ be the standard orthonormal basis for $\ell^2\F_2.$  Define the unitary operator $U$ on $\ell^2\F_2$ by  $U(\delta_x)=\delta_{\alpha(x)}$ for $x\in \F_2.$ Let $E_{\pm 1}$ denote the eigenspaces of $U.$  An orthonormal basis for $E_1$ is the set
 \begin{equation*} 
 \{ \delta_e \}\cup\Big\{ \frac{1}{\sqrt{2}}\delta_x+\frac{1}{\sqrt{2}}\delta_{\alpha(x)}:x\in W_{b,b^{-1}}  \Big\}
 \end{equation*}
and an orthonormal basis for $E_{-1}$ is the set
\begin{equation*} 
\Big\{ \frac{1}{\sqrt{2}}\delta_x-\frac{1}{\sqrt{2}}\delta_{\alpha(x)}:x\in W_{b,b^{-1}}  \Big\}.
\end{equation*}
\end{definition}
We build projections that approximately commute with $\lambda_a.$ This is a well-known procedure that is an easy case of Berg's technique.
\begin{definition} \label{def:NRdef} For very large integers $N$, let $R(N)$ be a function of $N$ that tends to infinity much slower than $N$,  for example suppose $R(N)4^{R(N)}\leq N^{1/4}.$ For ease of notation we will fix an $N$ and write $R$ instead of $R(N).$
Define
\begin{equation}
F=\{ (k,x): 0\leq k\leq N-1\textrm{ and }x\in B_R\cap W_{b,b^{-1}, e }  \}.
\end{equation}
For each element $(k,x)\in F$ define 
\begin{equation*}
\eta'(k,x)=\sqrt{\frac{k}{N}}\delta_{a^{k-N}x}+\sqrt{\frac{N-k}{N}}\delta_{a^kx}.
\end{equation*}
Let $P$ project onto $\textup{span}\{ \eta'(k,x):(k,x)\in  F   \}.$  
\end{definition}
\begin{remark} One sees that the commutator $[P,\lambda_a]$ has norm at most $N^{-{1/2}}.$  The projections $P$ therefore witness the quasidiagonality of $\lambda_a.$
\end{remark}
We now twist $P$ by a unitary halfway between $U$ and the identity.  
\begin{definition} \label{def:V}
Define $V\in B(\ell^2\F_2)$ by $V(\xi)=\xi$ if $\xi\in E_1$ and $V(\xi)=i\xi$ if $\xi\in E_{-1}$ (see Definition \ref{defn:Ueigen}). Define $Q=VPV^*$
\end{definition}
\begin{remark} The choice of $V$ in Definition \ref{def:V} is not really a choice.  Indeed, to make the calculations feasible we need to choose a unitary $W=\left(  \begin{array}{ll}  w_{11} & w_{12}\\ w_{21} & w_{22} \end{array} \right)$ subject to $\ell^2\F_2=E_1\oplus E_{-1}$ with $w_{ij}\in \C.$ A little arithmetic shows that to obtain the best upper bound for $\textup{qd}(\{  \lambda_a,\lambda_b \})$ one needs $|w_{11}+w_{12}+w_{21}+w_{22}|=|w_{11}+w_{12}-w_{21}-w_{22}|=\sqrt{2}$ (the maximum value).  Therefore we just chose the case with the most zeros to make things a little easier, but we can not obtain a better estimate with a different scalar matrix.
\end{remark}
We  record an orthonormal basis for the range of $Q$ that will be used later.
\begin{lemma} \label{lem:onb}  Set $A=(1+i)/2$ and $B=(1-i)/2.$  For $(0,x)\in F$ define
\begin{equation*}
\eta(0,x)=A\delta_x+B\delta_{\alpha(x)}
\end{equation*}
When $(k,x)\in F$ with $k\geq1$ set
\begin{equation*}
\eta(k,x)=\sqrt{\frac{k}{N}}(A\delta_{a^{k-N}x}+B\delta_{b^{k-N}\alpha(x)})+\sqrt{\frac{N-k}{N}}(A\delta_{a^kx}+B\delta_{b^k\alpha(x)})
\end{equation*}
Then the set $\{ \eta(k,x): (k,x)\in F  \}$ is an orthonormal basis for the range of $Q.$
\end{lemma}
\begin{proof} The proof follows from the equations $V(\delta_x)=A\delta_x+B\delta_{\alpha(x)}$ for all $x\in \F_2.$
\end{proof}
\subsubsection{Main Theorem}
We spend the rest of the section proving the following
\begin{theorem} \label{theorem:mainthm} Let $a,b\in \F_2$ be generators and $\lambda$ the left regular representation of $\F_2.$ Then $\max\{\V [\lambda_a,Q]  \V,\V [\lambda_b,Q]  \V\}\leq \frac{\sqrt{3}}{2}+N^{-1/9}.$ It follows that $\textup{qd}(\{\lambda_a,\lambda_b\})\in [1/2,\sqrt{3}/2].$
\end{theorem}
\begin{proof} It is clear that as $N$ tends to $\infty$ the projections $P$ tend strongly to the identity on $\ell^2\F_2$.  Therefore the projections $Q=VPV^*$ also tend strongly to the identity on $\ell^2\F_2.$ Therefore once we show $\max\{\V [\lambda_a,Q]  \V,\V [\lambda_b,Q]  \V\}\leq \frac{\sqrt{3}}{2}+N^{-1/9}$, the upper bound of the second claim will follow.  The lower bound was shown in \cite[Theorem 2.4]{Carrion13}.
Since $Q$ is a projection, we have 
\begin{align*}
\V Q\lambda_a-\lambda_a Q \V&=\max\{ \V Q\lambda_a(1-Q)\V,\V(1-Q)\lambda_aQ\V  \}  \\
&=\max\{ \V (1-Q)\lambda_{a^{-1}}Q\V,\V(1-Q)\lambda_aQ\V  \}  \\
&=\max\{ \V Q\lambda_{a^{-1}}Q-\lambda_{a^{-1}}Q\V,\V Q\lambda_aQ-\lambda_aQ\V  \}.   
\end{align*}
Let $\eta$ be  in the range of $Q.$  Set $\xi=\lambda_a\eta.$  Then
\begin{align}
\V  (Q\lambda_aQ-\lambda_aQ)\eta \V^2&=\V Q\xi-\xi \V^2 \label{eq:normest2} \\
&=\la Q\xi-\xi,Q\xi-\xi  \ra \notag\\
&=\V\xi\V^2-\la Q\xi,\xi \ra.\notag
\end{align}
We have a similar equality for any $x\in \{ a^{-1},b,b^{-1}  \}$ which provides the following 
\begin{align*}
\max\{\V[\lambda_a,Q]\V,\V[\lambda_b,Q]\V\}&=\max_{x\in \{ a,a^{-1},b,b^{-1}  \}}\sup_{\substack{\eta\in Q(\ell^2\F_2)\\ \V \eta \V=1}}1-\la Q\lambda_x\eta,\lambda_x\eta  \ra\\
&=\max_{x\in \{ a,a^{-1},b,b^{-1}  \}}\Big(1-\inf_{\substack{\eta\in Q(\ell^2\F_2)\\ \V \eta \V=1}}\la Q\lambda_x\eta,\lambda_x\eta  \ra\Big).\\
\end{align*}
Due to the symmetry in the definition of $Q$ (symmetry between $a$ and $b$ and also symmetry between $a$ and $a^{-1}$) it will become clear in the course of the proof that once we have shown $\la  Q\lambda_a\eta,\lambda_a\eta \ra\gtrsim \frac{1}{4}\V \eta \V^2 $ for any $\eta$ in the range of $Q$ that the same inequality holds with $a$ replaced by any $x\in\{ a,a^{-1},b,b^{-1}  \}.$  Therefore we focus solely on the quantity in (\ref{eq:normest2}).
Hence the proof will be complete once we prove the following 
\begin{equation}
\textbf{Claim:}\quad  \inf_{\xi\in \textup{Range}(\lambda_aQ)}\la Q\xi,\xi  \ra\geq (1-N^{-1/9})\frac{1}{4}\V \xi \V^2. \label{eq:claim}
\end{equation}
At this point  one can perform some ``back of the envelope" calculations alluded to in the introduction of this section.  Indeed, one sees fairly easily (but not without some labor) that for any $(k,x)\in F$ we have $\la Q\lambda_a\eta(k,x),\lambda_a\eta(k,x)  \ra\geq 1/4-\frac{R}{N}.$  In other words we have  (\ref{eq:claim}) for elements of an orthonormal basis for the range of $\lambda_aQ$, now we need to check that (\ref{eq:claim}) holds for everything in the span of these basis elements. 
\\\\
Let $\xi$ in the range of $\lambda_aQ$ be norm 1. We prove (\ref{eq:claim}).
\\\\
Lemma \ref{lem:onb} provides an orthonormal basis for $\lambda_aQ$ and we decompose $\xi$ as
\begin{equation}
\xi=\sum_{(k,x)\in F} \overline{\beta}_{k,x}\lambda_a\eta(k,x) \label{eq:defxi}
\end{equation}
for some scalars $\beta_{k,x}\in \C.$ We then have
\begin{equation} \label{eq:maincalc}
\la Q\xi,\xi  \ra=\sum_{(k,x)\in F}|\la \eta(k,x),\xi  \ra|^2=\sum_{(k,x)\in F}\Big| \sum_{(j,y)\in F}\beta_{j,y}\la \eta(k,x),\lambda_a\eta(j,y)  \ra  \Big|^2.
\end{equation}
Therefore we need to calculate all of the values $\la \eta(k,x),\lambda_a\eta(j,y)  \ra$ as both $(k,x)$ and $(j,y)$ vary through the set $F.$ This is where things get a little complicated and we need to put the information into a very organized form.  Tables \ref{table:table1} and \ref{table:table2} organize this data in a user friendly manner. Please see the Appendix Section \ref{sec:tables} for the tables and the explanation of the tables

The following lemma makes an observation that will cut the calculations in half.
\begin{lemma} \label{lem:coloredtable} Let 
\begin{equation*}
S=\{ (k,x)\in F: k+|x|\leq R-1 \textrm{ or }N-k+|x|\leq R-1 \}.  
\end{equation*}
Let $\zeta\in \textup{span}\{ \lambda_a\eta(k,x):(k,x)\in S  \}.$ Then
\begin{equation*}
\la Q \zeta,\zeta\ra  \geq\Big(1-\frac{1}{N^{1/4}}\Big)\V \zeta \V^2.
\end{equation*}
\end{lemma}
\begin{proof} The proof involves making another table from Tables \ref{table:table1} and \ref{table:table2}. This is easy, but space-consuming so we have placed the proof in the Appendix, Section \ref{appsec:lemmaproof}. 
\end{proof}
Consider our vector $\xi$ from (\ref{eq:defxi}).   Then decompose $\xi=\xi_1+\xi_2$ where $\xi_1\in \textup{span}\{ \lambda_a\eta(k,x):(k,x)\in S  \}$ and $\xi_2\in\textup{span}\{ \lambda_a\eta(k,x):(k,x)\in F\setminus S  \}.$ By Lemma \ref{lem:coloredtable} and using the fact that $\xi_1\perp \xi_2$ we have 
\begin{align*}
\la Q\xi,\xi  \ra&=\la Q(\xi_1+\xi_2),\xi_1+\xi_2 \ra\\
&=\la  Q\xi_1,\xi_1 \ra+2\textup{Re}(\la Q\xi_1,\xi_2 \ra)+\la  Q\xi_2,\xi_2 \ra\\
&\geq \Big(1-\frac{1}{N^{1/4}}\Big)\V \xi_1 \V^2-2\frac{1}{N^{1/8}}\V \xi_1 \V\V\xi_2\V+\la  Q\xi_2,\xi_2 \ra.
\end{align*}
The first two quantities effectively equal $\V \xi_1 \V^2$, so we only need to consider the quantity $\la Q\xi_2,\xi_2  \ra.$ 
Therefore \textbf{we may, without loss of generality, suppose that}
\begin{equation*}
 \xi\in \textup{span}\{ \lambda_a\eta(k,x):(k,x)\in F\setminus S  \}. 
 \end{equation*}
 We rewrite (\ref{eq:maincalc}) using our new assumption as
\begin{equation} \label{eq:maincalc2}
\la Q\xi,\xi  \ra=\sum_{(k,x)\in F }\Big| \sum_{(j,y)\in F\setminus S}\beta_{j,y}\la \eta(k,x),\lambda_a\eta(j,y)  \ra  \Big|^2
\end{equation}
We focus on proving the claim (\ref{eq:claim}).
We use Tables \ref{table:table1} and \ref{table:table2} to calculate (\ref{eq:maincalc2}). Please see the Appendix Section \ref{sec:tables} for the tables and their explanation.  We have eight distinct cases to consider.
  We introduce some notation to facilitate the calculations.
\begin{definition} For each integer $i=1,...,20$ let $G_i\subseteq F$ be the set consisting of those indices $(k,x)$ described in the box \emph{\textbf{i}} in Tables  \ref{table:table1} and \ref{table:table2}.
\end{definition}
We use the following identity often: For complex numbers $X,Y\in \C$ we have
\begin{equation} \label{eq:complexid}
|iX+Y|^2+|X+iY|^2=2|X|^2+2|Y|^2.
\end{equation}
In the final column of Tables \ref{table:table1} and \ref{table:table2} most of those constants multiplied by the terms involving $A$ or $B$ (for example $\sqrt{\frac{N-\ell}{N}}$) are very close to 1.  
In all of the calculations below we will implicitly use the relationship between $N$ and $R$ from Definition \ref{def:NRdef} to make the calculations easier. This is where all of the factors $(1+N^{1/4})$ come from in the following calculations. Informally, we just assume all of the constants like $\sqrt{\frac{N-\ell}{N}}$ are equal to  
1 and then cover up our deceit with the error term $1+N^{1/4}.$

The net effect of Lemma \ref{lem:coloredtable} on the following calculations is that we can completely ignore all of the red colored vectors, ignore the blue colored vectors unless the length condition is extremal and we can never ignore the black colored vectors.

\emph{Case 1.} Let $F_1\subseteq F$ be the following pairs $(k,x):$  
\begin{enumerate}
\item All of the pairs $(0,x)$ described in box \textbf{1}
\item All of the pairs $(1,x)$ described in box \textbf{9}
\item All of the pairs $(1,x)$ described in box \textbf{14} of the form $x=b^{-1}ay$ such that $y\in W_b.$
\end{enumerate}
Let $Q_{F_1}\leq Q$ be the projection onto $\textup{span}\{ \eta(k,x): (k,x)\in F_1  \}.$ By Tables \ref{table:table1} and \ref{table:table2} and our assumption on $\xi$ in (\ref{eq:maincalc2}) we obtain   
\begin{align*}
 &(1+N^{-1/4})\la Q_{F_1}\xi,\xi \ra\\
 &\geq\sum_{\substack{(0,x)\in G_1\\ |x|=R-1}}\Big|\frac{i}{2}\beta(0,b^{-1}\alpha(x))+\frac{1}{2}\beta(N-1,x)  \Big|^2+\sum_{\substack{(1,x)\in G_9\\ |x|=R-2}}\Big|-\frac{i}{2}\beta(N-1,b\alpha(x)) +\frac{1}{2}\beta(0,b^{-1}ax)  \Big|^2\\
 &+\sum_{\substack{(1,x)\in G_{14}\\ x=b^{-1}ay, y\in W_b\\  |y|=R-2}}\Big|\frac{i}{2}\beta(N-1,b\alpha(y)) +\frac{1}{2}\beta(0,x)  \Big|^2\\
 &=2\sum_{\substack{(0,x)\in G_1\\ |x|=R-1}}\Big|\frac{i}{2}\beta(0,b^{-1}\alpha(x))+\frac{1}{2}\beta(N-1,x)  \Big|^2+\sum_{\substack{(0,x)\in G_1\\ |x|=R-1}}\Big|\frac{i}{2}\beta(N-1,x) +\frac{1}{2}\beta(0,b^{-1}\alpha(x))  \Big|^2\\
&\geq \sum_{\substack{(0,x)\in G_1\\ |x|=R-1}}\Big|\frac{i}{2}\beta(0,b^{-1}\alpha(x))+\frac{1}{2}\beta(N-1,x)  \Big|^2+\sum_{\substack{(0,x)\in G_1\\ |x|=R-1}}\Big|\frac{i}{2}\beta(N-1,x) +\frac{1}{2}\beta(0,b^{-1}\alpha(x))  \Big|^2\\\\
&= \frac{1}{2}\sum_{\substack{(0,x)\in G_1\\ |x|=R-1}}|\beta(0,b^{-1}\alpha(x))|^2+|\beta(N-1,x)|^2.
\end{align*}
The last line follows by (\ref{eq:complexid}).
\\\\
\emph{Case 2.} Let $F_2\subseteq F$ be the following pairs $(k,x):$  
\begin{enumerate}
\item All of the pairs $(0,x)$ described in box \textbf{3}
\item All of the pairs $(1,x)$ described in box \textbf{12}
\item All of the pairs $(1,x)$ described in box \textbf{14} of the form $x=b^{-1}ay$ such that $y\in W_{b^{-1}}.$
\end{enumerate}
Let $Q_{F_2}\leq Q$ be the projection onto $\textup{span}\{ \eta(k,x): (k,x)\in F_2  \}.$
By the same calculations as in Case 1 we have
\begin{equation*}
(1+N^{-1/4})\la Q_{F_2}\xi,\xi \ra\geq \frac{1}{2}\sum_{\substack{(0,x)\in G_3\\ |x|=R-1}}|\beta(0,b^{-1}\alpha(x))|^2+|\beta(N-1,x)|^2.
\end{equation*}
\emph{Case 3.} Let $F_3\subseteq F$ be the following pairs $(k,x):$  
\begin{enumerate}
\item All of the pairs $(0,x)$ described in box \textbf{5}
\item All of the pairs $(1,x)$ described in box \textbf{15}
\item All of the pairs $(1,x)$ described in box \textbf{14} of the form $x=b^{-1}y$ such that $y\in W_{a^2}.$
\end{enumerate}
Let $Q_{F_3}\leq Q$ be the projection onto $\textup{span}\{ \eta(k,x): (k,x)\in F_3  \}.$
By the same calculations as in Case 1 we have
\begin{equation*}
(1+N^{-1/4})\la Q_{F_3}\xi,\xi \ra\geq \frac{1}{2}\sum_{\substack{(0,x)\in G_5\\ |x|=R-1}}|\beta(0,b^{-1}\alpha(x))|^2+|\beta(N-1,x)|^2.
\end{equation*}
\emph{Case 4.} Let $F_4\subseteq F$ be the following pairs $(k,x):$  
\begin{enumerate}
\item All of the pairs $(0,x)$ described in box \textbf{7}
\item All of the pairs $(k,x)$ described in box \textbf{18}
\item All of the pairs $(1,x)$ described in box \textbf{14} of the form $x=b^{-1}a^{-k}y$ such that $y\in W_{b,b^{-1},e}$ with $|y|+k+1=R$ and $k\geq1.$
\item All of the pairs $(1,x)$ described in box \textbf{14} of the form $x=b^{-k-1}y$ such that $y\in W_{a,a^{-1},e}$ with $|y|+k+1=R$ and $k\geq1.$
\end{enumerate}
Let $Q_{F_4}\leq Q$ be the projection onto $\textup{span}\{ \eta(k,x): (k,x)\in F_4  \}.$
\\\\
This case is more complicated than the last three. We are summing over one type of elements from \textbf{7} and \textbf{18} while summing over two distinct types of elements of \textbf{14}. All four of these need to come together to produce a satisfying estimate. By our assumption on $\xi$ in (\ref{eq:maincalc2}) we have
\begin{equation*}
\la Q_{F_4}\xi,\xi \ra= \sum_{\substack{(0,x)\in G_7\\ |x|=R-1}}|\la  \eta(0,x),\xi \ra|^2+\sum_{\substack{(k,x)\in G_{18}\\ N-k+|x|=R-1}}|\la \eta(k,x),\xi \ra|^2+\sum_{(1,x)\in F_4\cap G_{14}}|\la \eta(1,x),\xi \ra|^2
\end{equation*}
We will calculate each of the above sums separately. We have
\begin{align*}
&(1+N^{-1/4})\sum_{\substack{(0,x)\in G_7\\ |x|=R-1}}|\la  \eta(0,x),\xi \ra|^2 \\
&\geq\frac{1}{4}\sum_{\ell=N-R+1}^{N-1}\sum_{\substack{  y\in W_{a,a^{-1},e}\\ |y|+N-\ell=R-1 }}|i\beta_{0,b^{-1}a^{\ell-N}\alpha(y)}+\beta_{N-1,b^{\ell-N}y}+\beta_{0,b^{\ell-N-1}y}-i\beta_{\ell-1,\alpha(y)}|^2 
\end{align*}
Then
\begin{align*}
&(1+N^{-1/4})\sum_{\substack{(k,x)\in G_{18}\\ N-k+|x|=R-1}}|\la  \eta(k,x),\xi \ra|^2\\
&\geq\frac{1}{4}\sum_{\ell=N-R+1}^{N-1}\sum_{\substack{  y\in W_{a,a^{-1},e}\\ |y|+N-\ell=R-1 }}|\beta_{0,b^{-1}a^{\ell-N}\alpha(y)}-i\beta_{N-1,b^{\ell-N}y}+i\beta_{0,b^{\ell-N-1}y}+\beta_{\ell-1,\alpha(y)}|^2 
\end{align*}
Now by (\ref{eq:complexid}) applied to the above two sums we obtain
\begin{align*}
&(1+N^{-1/4})\sum_{(k,x)\in G_7\cup G_{18}}|\la  \eta(k,x),\xi \ra|^2 \\
&\geq\frac{1}{2}\sum_{\ell=N-R+1}^{N-1}\sum_{\substack{  y\in W_{a,a^{-1},e}\\ |y|+N-\ell=R-1 }}|i\beta_{0,b^{-1}a^{\ell-N}\alpha(y)}+\beta_{N-1,b^{\ell-N}y}|^2+|i\beta_{0,b^{\ell-N-1}y}+\beta_{\ell-1,\alpha(y)}|^2
\end{align*}
Then
\begin{align*}
&(1+N^{-1/4})\sum_{(1,x)\in G_{14}\cap F_4}|\la  \eta(1,x),\xi \ra|^2\\
&\geq\frac{1}{4}\sum_{\ell=N-R+1}^{N-1}\sum_{\substack{  y\in W_{a,a^{-1},e}\\ |y|+N-\ell=R-1 }}|\beta_{0,b^{-1}a^{\ell-N}\alpha(y)}+i\beta_{N-1,b^{\ell-N}y}|^2+|\beta_{0,b^{\ell-N-1}y}+i\beta_{\ell-1,\alpha(y)}|^2\\
\end{align*}
We then apply (\ref{eq:complexid}) again to the above two sums to obtain
\begin{align*}
&(1+N^{-1/4})\la Q_{F_4}\xi,\xi \ra\\
&\geq \frac{1}{2}\sum_{\ell=N-R+1}^{N-1}\sum_{\substack{  y\in W_{a,a^{-1},e}\\ |y|+N-\ell=R-1 }}|i\beta_{0,b^{-1}a^{\ell-N}\alpha(y)}+\beta_{N-1,b^{\ell-N}y}|^2+|i\beta_{0,b^{\ell-N-1}y}+\beta_{\ell-1,\alpha(y)}|^2\\
&+\frac{1}{4}\sum_{\ell=N-R+1}^{N-1}\sum_{\substack{  y\in W_{a,a^{-1},e}\\ |y|+N-\ell=R-1 }}|\beta_{0,b^{-1}a^{\ell-N}\alpha(y)}+i\beta_{N-1,b^{\ell-N}y}|^2+|\beta_{0,b^{\ell-N-1}y}+i\beta_{\ell-1,\alpha(y)}|^2\\
&\geq \frac{1}{2}\sum_{\ell=N-R+1}^{N-1}\sum_{\substack{  y\in W_{a,a^{-1},e}\\ |y|+N-\ell=R-1 }}|\beta_{0,b^{-1}a^{\ell-N}\alpha(y)}|^2+|\beta_{N-1,b^{\ell-N}y}|^2+|\beta_{0,b^{\ell-N-1}y}|^2+|\beta_{\ell-1,\alpha(y)}|^2.
\end{align*}
\emph{Case 5.} Let $F_5\subseteq F$ be the following pairs $(k,x):$  
\begin{enumerate}
\item All of the pairs $(0,x)$ described in box \textbf{8}
\item All of the pairs $(k,x)$ described in box \textbf{19}
\end{enumerate}
Let $Q_{F_5}\leq Q$ be the projection onto $\textup{span}\{ \eta(k,x): (k,x)\in F_5  \}.$ We have
\begin{align*}
&(1+N^{-1/4})\la Q_{F_5}\xi,\xi \ra=(1+N^{-1/4})\Big[\sum_{(0,x)\in G_8}|\la \eta(0,x),\xi  \ra|^2+\sum_{(k,x)\in G_{19}}|\la \eta(k,x),\xi  \ra|^2\Big]\\
&\geq\frac{1}{4}\sum_{\ell=N-R}^{N-1}\sum_{\substack{  y\in W_{a,a^{-1},e}\\ |y|+N-\ell=R }}|\beta_{N-1,b^{\ell-N}y}-i\beta_{\ell-1,\alpha(y)}|^2\\
&+\frac{1}{4}\sum_{\ell=N-R}^{N-1}\sum_{\substack{  y\in W_{a,a^{-1},e}\\ |y|+N-\ell=R }}|-i\beta_{N-1,b^{\ell-N}y}+\beta_{\ell-1,\alpha(y)}|^2\\
&=\frac{1}{2}\sum_{\ell=N-R}^{N-1}\sum_{\substack{  y\in W_{a,a^{-1},e}\\ |y|+N-\ell=R }}|\beta_{N-1,b^{\ell-N}y}|^2+|\beta_{\ell-1,\alpha(y)}|^2
\end{align*}
Again, the last line follows from (\ref{eq:complexid}).
\\\\
\emph{Case 6.} Let $F_6\subseteq F$ be the following pairs $(k,x):$  
\begin{enumerate}
\item All of the pairs $(1,x)$ described in box \textbf{11}
\item All of the pairs $(k,x)$ described in box \textbf{17}
\end{enumerate}
Let $Q_{F_6}\leq Q$ be the projection onto $\textup{span}\{ \eta(k,x): (k,x)\in F_6  \}.$ We have
\begin{align*}
&(1+N^{-1/4})\la Q_{F_6}\xi,\xi \ra=(1+N^{-1/4})\Big[\sum_{(1,x)\in G_{11}}|\la \eta(1,x),\xi  \ra|^2+\sum_{(k,x)\in G_{17}}|\la \eta(k,x),\xi  \ra|^2\Big]\\
&\geq\frac{1}{4}\sum_{\ell=1}^R\sum_{\substack{  y\in W_{a,a^{-1},e}\\ |y|+\ell=R }}|\beta_{0,b^\ell y}+i\beta_{\ell,\alpha(y)}|^2\\
&+\frac{1}{4}\sum_{\ell=1}^R\sum_{\substack{  y\in W_{a,a^{-1},e}\\ |y|+\ell=R }}|i\beta_{0,b^\ell y}+\beta_{\ell,\alpha(y)}|^2\\
&=\frac{1}{2}\sum_{\ell=1}^R\sum_{\substack{  y\in W_{a,a^{-1},e}\\ |y|+\ell=R }}|\beta_{0,b^\ell y}|^2+|\beta_{\ell,\alpha(y)}|^2
\end{align*}
Again, the last line follows from (\ref{eq:complexid}).
\\\\
\emph{Case 7.} Let $F_7\subseteq F$ be all of the pairs described in boxes \textbf{2,4,6,10,13,16}. By our assumption on $\xi$ in (\ref{eq:maincalc2}) we have
\begin{equation*}
(1+N^{-1/4})\la Q_{F_7}\xi,\xi  \ra\geq\frac{1}{4}\sum_{x\in W_b\textrm{, }|x|=R}|\beta_{N-1,x}|^2
\end{equation*}
\emph{Case 8.} Let $F_8$ be all of the pairs described in box \textbf{20.} Then
\begin{equation*}
\la Q_{F_8}\xi,\xi  \ra=\frac{1}{4}\sum_{(k,x)\in G_{20}}s(k,N)|\beta_{k-1,x}|^2\geq (1-N^{-1})\frac{1}{4}\sum_{(k,x)\in G_{20}}|\beta_{k-1,x}|^2.
\end{equation*}
Notice that every pair $(k,x)\in F\setminus S$ is represented in the last line of exactly one of the 8 cases above. This is no surprise as this is how we arranged it--all of the red and blue colored vectors match up with at least one of the above cases.  Therefore
\begin{align*}
(1+N^{-1/4})\la Q\xi,\xi \ra&= (1+N^{-1/4})\sum_{i=1}^8\la Q_{F_i}\xi,\xi \ra\\
&\geq\frac{1}{4}\sum_{(k,x)\in F\setminus S}|\beta_{k,x}|^2\\
&=\frac{1}{4}.
\end{align*}
Therefore, together with Lemma \ref{lem:coloredtable} we proved claim (\ref{eq:claim}) and hence Theorem \ref{theorem:mainthm}.
\end{proof}

\section{Free groups, operator spaces and quasidiagonality} \label{sec:posarbsmall}
We refer the reader to Brown and Guentner's paper \cite{Brown13} for information on $\ell^p$-representations of discrete groups. See Definition \ref{def:lprep} for the definitions and the remarks following for the basic results used in this section.
\begin{definition} \label{def:leftsubs} Throughout the section we fix a generating set $a_1,...,a_d$ for $\F_d.$  For each $x\in \F_d$ we write $\delta_x\in \C[\F_d]$ as the function that is 1 at $x$ and 0 otherwise. 
\end{definition}
\subsection{Positive and arbitrarily small modulus of quasidiagonality}
In this section (see Theorem \ref{thm:arbclose0}) we show that for each $\ve>0$ there is a set of unitaries $\Omega$ such that $\textup{qd}(\Omega)\in (0,\ve).$ 

At the center of all of calculations are Haagerup's positive definite functions  from his seminal paper \cite{Haagerup79}.  For each $0<r<1$ Haagerup showed that the function $\phi_r(t)=r^{|t|}$ is positive definite on $\F_d.$ Later Pytlik and Szwarc showed how to choose the GNS representations associated with these functions in a continuous manner.  More specifically, in \cite{Pytlik86} for each $z\in \C$ with $|z|<1$ they constructed a uniformly bounded representation $\pi_z$ of $\F_d$ on $B(\ell^2\F_d)$ such that $z\mapsto \pi_z(t)$ is continuous for each $t\in \F_d.$  Moreover when $z\in (0,1)$, the representation $\pi_z$ is unitary and $z^{|t|}=\la \pi_z(t) \delta_e,\delta_e\ra$ for $t\in \F_d.$ 

We now specifically recall their representations (see \cite[Equation (2), Page 291]{Pytlik86}).  Let $0\leq z \leq 1.$ Define the unitary representation $\pi_z:\F_d\rightarrow B(\ell^2\F_d)$ by
\begin{equation} \label{eq:pizdef}
\pi_z(a_i)(\delta_x)=\left\{ \begin{array}{cl} \delta_{a_ix}   &  \textrm{ if }x\not\in \{ e,a_i^{-1} \}\\
                                                                   \sqrt{1-z^2}\delta_{a_i}+z\delta_e & \textrm{ if } x=e\\ 
                                                                   -z\delta_{a_i}+\sqrt{1-z^2}\delta_e & \textrm{ if } x=a_i^{-1}\\ \end{array} \right.
\end{equation}
Notice that $\pi_0=\lambda$ is the left regular representation, while $\pi_1=t\oplus \widetilde{\lambda}$ where $\widetilde{\lambda}$ is weakly equivalent to $\lambda$ and $t$ is the trivial representation. Furthermore by \cite[Remark 2.4]{Pytlik86}, one recovers Haagerup's positive definite functions as $t\mapsto \la  \pi_z(t)\delta_e,\delta_e \ra=z^{|t|}.$

The Pytlik and Szwarc construction not only connects $\lambda$ continuously to $\pi_1$ but it effectively connects $\lambda$ to any representation that weakly contains $\lambda$!  Indeed suppose that $\sigma$ is any representation of $\F_d$ such that $\lambda\prec \sigma.$  By Fell's absorption principle $\pi_0\otimes \sigma$ is weakly equivalent to $\lambda$ while $\pi_1\otimes \sigma=\sigma\oplus (\widetilde{\lambda}\otimes\sigma)$ is weakly equivalent to $\sigma.$  We exploit this fact in a couple contexts.
We start by pointing out the following
\begin{lemma} \label{lem:z1estimate} Let $1\leq i\leq d.$  Then $\V \pi_z(a_i)-\pi_1(a_i)  \V=\sqrt{2-2z}.$
\end{lemma}
\begin{proof} One sees that
\begin{equation*}
\V \pi_z(a_i)-\pi_1(a_i)  \V_{B(\ell^2\F_d)}=\left\V \left(  \begin{array}{ll}  z-1 & \sqrt{1-z^2}\\ \sqrt{1-z^2} & 1-z \end{array}  \right)  \right\V_{M_2(\C)}=\sqrt{2-2z}.
\end{equation*} 
\end{proof}
\begin{theorem} \label{thm:arbclose0} Let $2\leq p<\infty$ and $d\geq 2$ an integer. Let $\sigma_p$ be a faithful, essential representation of $C^*_{\ell^p}(\F_d)$, i.e. the range of $\sigma_p$ has trivial intersection with the compact operators. Then
\begin{equation*}
 0<\textup{qd}(\{ \sigma_p(a_i): 1\leq i\leq d  \})\leq \sqrt{2-2(2d-1)^{-\frac{1}{p}}}.
 \end{equation*}
\end{theorem}
\begin{proof} Ruan and Wiersma showed in \cite[Corollary 5.3]{Ruan16} that $C^*_{\ell_p}(\F_d)$ is not a quasidiagonal C*-algebra.  Therefore by Voiculescu's theorem \cite{Voiculescu76}, the representation $\sigma_p$ is not a quasidiagonal representation. Since $\{ \sigma_p(a_i): 1\leq i\leq d  \}$ generates $C^*_{\ell^p}(\F_d)$, it follows that $\textup{qd}(\{ \sigma_p(a_i): 1\leq i\leq d  \})>0.$ We now focus on the other inequality.

Let $\sigma_\infty$  be a faithful essential representation of $C^*(\F_d)$.  The map $\sigma_p$  defines a (non-faithful) representation of $C^*(\F_d).$ Since $\pi_1$ contains the trivial representation, $\pi_1\otimes \sigma_\infty$ is a faithful essential representation of $C^*(\F_d).$ Therefore $\sigma_p\oplus (\pi_1\otimes \sigma_\infty)$ is a faithful, essential representation of $C^*(\F_d).$ By \cite{Choi80}, the C*-algebra $C^*(\F_d)$ is residually finite dimensional and therefore quasidiagonal.  By Voiculescu's theorem \cite{Voiculescu76}, the representation $\sigma_p\oplus (\pi_1\otimes \sigma_\infty)$ is quasidiagonal. 

Let $H$ denote the Hilbert space on which the representation $\pi:=\sigma_p\oplus (\pi_1\otimes \sigma_\infty)$  is defined.
Let $P_n\in B(H)$ be an increasing sequence of finite rank projections that converge strongly to the identity such that $\lim_{n\rightarrow\infty}\V [x,P_n] \V=0$ for all $x\in \pi(C^*(\F_d)).$

For any $z<(2d-1)^{-\frac{1}{p}}$, we have $z^{|\cdot|}\in \ell^p(\F_d),$ i.e. $\pi_z$, as defined in Lemma \ref{lem:z1estimate}, is a representation of $C^*_{\ell^p}(\F_d).$ By Lemma \ref{lem:z1estimate}, we have 
\begin{align*}
\max_{1\leq i\leq d} \V \sigma_p\oplus (\pi_1\otimes \sigma_\infty)(a_i)- \sigma_p\oplus (\pi_z\otimes \sigma_\infty)(a_i)\V & = \max_{1\leq i\leq d} \V \pi_1(a_i)-\pi_z(a_i)  \V\\
&\leq \sqrt{2-2z}.
\end{align*}
For any bounded operators $T,S$ and orthogonal projection $P$ acting on a Hilbert space we have
\begin{equation*}
 \V [P,T]  \V\leq \V T-S \V+\V  [P,S] \V. \label{eq:commest}
 \end{equation*}
 Therefore for each $i=1,...,d$ we have
 \begin{equation*}
 \limsup_{n\rightarrow\infty}\V [P_n,\sigma_p\oplus (\pi_z\otimes \sigma_\infty)(a_i)] \V\leq \sqrt{2-2z}.
 \end{equation*}
Therefore $\textup{qd}(\{ \sigma_p\oplus (\pi_z\otimes \sigma_\infty)(a_i):1\leq i\leq d  \})\leq \sqrt{2-2z}.$
By Voiculescu's theorem \cite{Voiculescu76}, the representations $\sigma_p$ and $\sigma_p\oplus (\pi_z\otimes \sigma_\infty)$ are approximately unitarily equivalent, from which it follows that $\textup{qd}(\{ \sigma_p(a_i):1\leq i\leq d  \})\leq \sqrt{2-2z}.$ Letting $z\nearrow (2d-1)^{-\frac{1}{p}}$ we obtain the conclusion.
\end{proof}
\begin{remark} Although not explicitly used, the fact that $\F_d$ has the Haagerup approximation property is the underlying reason why Theorem \ref{thm:arbclose0} works out.  Let 
\begin{equation*}
D=\bigcup_{1\leq p<\infty}\ell^p(\F_d).
\end{equation*}
Then by a combination of the proof of \cite[Theorem 3.2]{Brown13} with \cite[Lemma 1.2]{Haagerup79} it follows that the $D$-norm on $\C[\F_d]$ is equal to the $\ell^\infty(\F_d)$ norm on $\C[\F_d].$  This is why sliding the parameter $p$ towards $\infty$ preserved the norm.
\end{remark}
\subsection{Operator space structure} We show that for any finite dimensional subspace $Y\subseteq \C[\F_d]$, the operator space structure of $Y\subseteq C^*_{\ell^p}[\F_d]$ varies continuously with $p$ (see Corollary \ref{cor:BMdistancearb}). We obtain explicit bounds in the case that $Y$ is the span of the unit ball in $\F_2.$ We refer the reader to either monograph \cite{Effros00,Pisier03} for the basic principles of operator space theory. 
\begin{lemma} \label{lem:Yzdef} For $0\leq z\leq 1$ define the operator space 
\begin{equation*}
Y_z=\textup{span}\{ \pi_z(a_i): 1\leq i\leq d   \}\subseteq B(\ell^2\F_d).
\end{equation*}
  Let $T:Y_z\rightarrow Y_1$ be defined by $T(\pi_z(t))=\pi_1(t).$ Then $\V T \V_{cb}\leq 1+d\sqrt{2-2z}.$
\end{lemma}
\begin{proof} We use a simplified version of  Pisier's perturbation lemma \cite[Proposition 2.13.2]{Pisier03}. Since our conclusion is slightly different we include the brief argument.
\\\\
 Each $\pi_z(a_i)$ differs from $\lambda_{a_i}$ by a rank two operator.  Therefore $\pi_z(a_i)\in C^*_r(\F_d)+K(\ell^2\F_d).$ Let $\sigma:  C^*_r(\F_d)+K(\ell^2\F_d)\rightarrow C^*_r(\F_d)$ be the quotient homomorphism and let $\tau$ be the canonical trace on $C^*_r(\F_d)$. For $i=1,...,d$ define the norm one functional on $Y_z$ by
\begin{equation*}
\omega_i(x)=\tau(\lambda(t^{-1})\sigma(x)).
\end{equation*}
Then $\{ \pi_z(a_i),\omega_i  \}_{1\leq i\leq d}$ forms a biorthogonal system for $Y_z.$    Define
\begin{equation*}
\delta(\pi_z(x))=\sum_{i=1}^d\omega_i(x)(\pi_1(a_i)-\pi_z(a_i)).
\end{equation*}
Since each $\omega_i$ is rank 1 and norm 1, it is completely bounded with cb norm 1.  Therefore $\V \delta \V_{cb}\leq d\sqrt{2-2z}$ by Lemma \ref{lem:z1estimate}.  Since $T=\textup{id}_{Y_z}+\delta$ the proof is complete.

\end{proof}
\begin{definition}\label{definition:W1}  Let $\ww=\textup{span}\{ \delta_{a_i}:1\leq i\leq d \}\subseteq \C[\F_d].$ For each $2\leq p\leq \infty$, let $\ww_p\subseteq C^*_{\ell^p}(\F_d)$ be $\ww$ equipped with the operator space structure inherited from $C^*_{\ell^p}(\F_d).$
\end{definition}
We recall the completely bounded version of Banach-Mazur distance.
\begin{definition} \label{def:dcb} Let $X$ and $Y$ be operator spaces.  Define
\begin{equation*}
d_{cb}(X,Y)=\inf\{ \V \phi \V_{cb}\V   \phi^{-1} \V_{cb}: \phi:X\rightarrow Y \textrm{ is a linear isomorphism } \}.
\end{equation*}
\end{definition}
\begin{theorem} \label{thm:BMdistance}   Let $2\leq p<q\leq \infty$ and $T_{p,q}:\ww_p\rightarrow \ww_q$ be the identity map.  Then
\begin{equation*}
 \V  T_{p,q} \V_{cb}\leq 1+d\sqrt{2(1-(2d-1)^{\frac{p-q}{pq}})}.
\end{equation*}
Since $T_{p,q}^{-1}$ is the restriction of a *-homomorphism from $C^*_{\ell^q}(\F_d)$ onto $C^*_{\ell^p}(\F_d)$ we have $\V T_{p,q}^{-1} \V_{cb}=1.$ Therefore
\begin{equation*}
d_{cb}(\ww_p,\ww_q)\leq 1+d\sqrt{2(1-(2d-1)^{\frac{p-q}{pq}})}.
\end{equation*}
\end{theorem}  
\begin{proof}
Let $r=\frac{pq}{q-p}$, i.e. $\frac{1}{r}+\frac{1}{q}=\frac{1}{p}.$  Let $\phi\in\ell^q(\F_d)$ be a normalized positive definite function and $\ve>0.$ Set $s_\ve=(2d-1)^{-\frac{1}{r}-\ve}.$  Let $\psi_{s_\ve}(t)=s_\ve^{|t|}.$  Then $\psi_{s_\ve}\in \ell^r(\F_d)$ and the GNS representation associated with $\psi_{s_\ve}$ is unitarily equivalent to the representation $\pi_{s_\ve}$ from Definition \ref{eq:pizdef} by \cite[Remark 2.4(2)]{Pytlik86}.

By the generalized H\"older inequality, we have  $\phi\psi_{s_\ve}\in \ell^p(\F_d).$ Let $\pi_\phi$ denote the GNS representation of $\F_d$ associated with $\phi.$ The GNS representation associated with $\phi\psi_{s_\ve}$ is unitarily equivalent to the tensor product representation $\pi_\phi\otimes\pi_{s_\ve}.$
Let $x\in \ww_p\otimes M_n.$ Since $\pi_\phi\otimes \pi_{s_\ve}$ is an $\ell^p$-representation we have
\begin{equation*}
\V x \V_{\ww_p\otimes M_n}  \geq \V  \pi_\phi\otimes \pi_{s_\ve}\otimes \textup{id}_{M_n}(x) \V.
\end{equation*}
By Lemma \ref{lem:Yzdef} we have
\begin{align*}
 \V  \pi_\phi\otimes \pi_{s_\ve}\otimes \textup{id}_{M_n}(x) \V&\geq  \Big(1+d\sqrt{2-2s_\ve} \Big)^{-1} \V  \pi_\phi\otimes \pi_1\otimes \textup{id}_{M_n}(x)    \V\\
&\geq \Big(1+d\sqrt{2-2s_\ve} \Big)^{-1} \V  \pi_\phi\otimes \textup{id}_{M_n}(x)    \V.
\end{align*}
The last line follows because $\pi_1$ contains the trivial representation (see (\ref{eq:pizdef})).

For any normalized positive definite function $\psi$, let $\pi_\psi$ denote the GNS representation associated with $\psi.$ It follows from  \cite[Theorem 3.4]{Okayasu14} that for $y\in \C[\F_d]$ and $2\leq q<\infty$ we have
\begin{equation*}
\V y \V_{C^*_{\ell^q}(\F_d)}=\sup\{ \V  \pi_\psi(y) \V: \psi\in \ell^q(\F_d) \textrm{ and }\psi \textrm{ is normalized positive definite} \}.
\end{equation*}
Therefore 
\begin{equation*}
\V x \V_{\ww_p\otimes M_n}\geq \Big(1+d\sqrt{2-2s_\ve}) \Big)^{-1}\sup_{\phi}\V  \pi_\phi\otimes \textup{id}_{M_n}(x)    \V= \Big(1+d\sqrt{2-2s_\ve}) \Big)^{-1}\V x \V_{\ww_q\otimes M_n}
\end{equation*}
Finally, let $\ve\rightarrow0$ to complete the proof.

\end{proof}
\begin{remark} We would like to point out that the natural map between the operator spaces $Y_{(2d-1)^{-\frac{1}{p}}}$ of Lemma \ref{lem:Yzdef} and $\ww_p$ of Definition \ref{definition:W1} is not a complete isometry.  Indeed, it is clear from the definition of $\pi_z$ that the C*-algebra generated by $\pi_z$ for each $z$ is contained in $C_r^*(\F_d)+K(\ell^2\F_d),$ in particular it is an exact C*-algebra by \cite{Choi79}. If the canonical map from $\ww_p$ onto $Y_{(2d-1)^{-\frac{1}{p}}}$ was a complete isometry, then it would extend to an isomorphism between $C^*_{\ell^p}(\F_d)$ and $C^*(\pi_z(\F_d)).$  But Ruan and Wiersma showed that the C*-algebra $C^*_{\ell^p}(\F_d)$ is not exact  \cite[Theorem 4.5]{Ruan16} whenever $p>2$.

We find it fascinating  that one can use the representations $\pi_z$ to obtain an estimate for $d_{cb}(\ww_p,\ww_q)$  even though the representations $\pi_z$ necessarily ``miss" some of the operator space structure of the spaces $\ww_p.$ 
\end{remark}
\begin{remark} In Theorem \ref{thm:BMdistance} we focused on the subspace of $\C[\F_d]$ spanned by group elements of length less than or equal to 1.  We did this to obtain a reasonable estimate on the cb norm.  By using the exact same methods as above we obtain the following.
\end{remark}
\begin{corollary} \label{cor:BMdistancearb} Let $Y\subseteq \C[\F_d]$ be any finite dimensional subspace. Let $p\in[2,\infty]$  and let $Y_p$ be $Y$ equipped the operator space structure inherited from $C^*_{\ell^p}(\F_d).$  Then 
\begin{equation*}
\lim_{q\rightarrow p} d_{cb}(Y_q,Y_p)=1.
\end{equation*}
\end{corollary}
Theorem \ref{thm:BMdistance} raises the following 
\begin{question} Can we obtain a more familiar description of the operator space structure of $\ww_p$?  What about just the Banach space structure? It is well-known that $\ww_\infty$ is completely isometric with $\textup{MAX}(\ell^1_d)$ (see \cite[Theorem 9.6.1]{Pisier03}).  On the other hand, by \cite{Pisier98}, $\ww_2$ is completely isomorphic with the row intersect column operator space $R_d\cap C_d$ with isomorphism constant bounded by 2 (in particular independent of $d$).  At the Banach space level, $R_d\cap C_d$ is a $d$-dimensional Hilbert space and $\textup{MAX}(\ell^1_d)$ is of course isometric to $\ell^1_d$ as a Banach space.  

It is tempting to guess that $\ww_p$ is Banach space isomorphic (independent of $d$) to $\ell^q_d$ where $\frac{1}{p}+\frac{1}{q}=1$, in fact there is already evidence to support this guess.  Okayasu proved in \cite[Lemma 3.3]{Okayasu14} that the natural map from   $\ell^q_d$ to $\ww_p$ has norm bounded by 2. We  failed to obtain a reasonable bound for the norm of the map from $\ww_p$ to $\ell^q_d.$ We also do not have a good guess for the operator space structure for $\ww_p.$  
\end{question}
\appendix
\section{Appendices}
\subsection{Tables} \label{sec:tables}
Let us explain the organization of Tables \ref{table:table1} and \ref{table:table2}. The left hand column titled ``$\xi_1\in Q(\ell^2\F_2)$" lists all the basis vectors from Lemma \ref{lem:onb} partitioned into classes that facilitate the calculation of $\la \eta(k,x),\lambda_a\eta(j,y)  \ra.$  The next column titled ``$\xi_2\in \lambda_aQ(\ell^2\F_2)$" lists all of those basis vectors in $\lambda_aQ(\ell^2\F_2)$ that are not orthogonal to the corresponding $\eta(k,x)$ from the first column.  The final column, titled $\la  \xi_1,\xi_2 \ra$ is simply the calculation of the inner product.  Recall that $A=(1+i)/2$ and $B=(1-i)/2.$  Therefore $A\overline{B}=\frac{i}{2}$ and $\overline{A}B=-\frac{i}{2}.$ We define 
\begin{equation*}
s(k,N)=\frac{1}{4N}(\sqrt{k(k-1)}+\sqrt{(N-k)(N-k+1)}).
\end{equation*}
Notice that $s(k,N)\geq \frac{1}{4N}(k-1+N-k)=\frac{N-1}{4N}.$
Technically these are the only features of the table needed for the proof of Theorem \ref{theorem:mainthm}.  There are two other features that make the calculations easier.

The table is color coded. We have colored blue or red all of the vectors $\lambda_a\eta(k,x)$ where $(k,x)\in S=\{ (k,x)\in F: k+|x|\leq R-1 \textrm{ or }N-k+|x|\leq R-1 \}$ (see Lemma \ref{lem:coloredtable} for the significance of $S$). There is a difference between the blue and red vectors.  The vectors colored red always correspond to elements of $S$ while the vectors colored blue only correspond to elements of $S$ in the case of non-extremal length conditions.  For example consider box \textbf{1.}  For the red vector $\color{red} \lambda_a\eta(0,a^{-1}\alpha(x))$ under the conditions of box \textbf{1} we always have $(0,a^{-1}\alpha(x))\in S.$  On the other hand, for the blue vector $\color{blue}\lambda_a\eta(0,b^{-1}\alpha(x))$ we have $(0,b^{-1}\alpha(x))\in S$ if and only if $|x|<R-1.$

Also beneath each bold box there are sometimes italicized numbers.  For example in bold box \textbf{7} there are italicized numbers \emph{14} and \emph{18}.  This means that there are vectors $\lambda_a\eta(k,x)$ such that $\la \xi_1,\lambda_a\eta(k,x) \ra\neq0$ and  $\la \xi_2,\lambda_a\eta(k,x) \ra\neq0$ for vectors $\xi_1$ from box 1 and $\xi_2$ from box 14 (or 18).  The only exception is box \textbf{14} which has common vectors with too many boxes to list.  These italicized numbers are technically not necessary but they make the calculations of the eight cases in the proof of Theorem \ref{theorem:mainthm} much simpler.

\begin{table}
\caption{} \label{table:table1}
\begin{tabular}{|l|l|ll|ll|} 
\hline
&$\xi_1\in Q(\ell^2\F_2)$&$\xi_2\in \lambda_aQ(\ell^2\F_2)$&&$\la \xi_1,\xi_2 \ra$&\\
\hline
&$\eta(0,e)$& $\color{red} \lambda_a\eta(N\textrm{-}1,e)$ & $\color{red} \lambda_a\eta(0,b^{-1})$  &$\sqrt{\frac{N\textrm{-}1}{N}}\overline{A}$&$\overline{B}$\\
\hline
&$\eta(0,b)$& $\color{red} \lambda_a\eta(N\textrm{-}1,b)$ & $\color{red} \lambda_a\eta(0,e)$  &$\sqrt{\frac{N\textrm{-}1}{N}}|A|^2$&$B$\\
&&$\color{red} \lambda_a\eta(0,b^{-1}a)$&&$A\overline{B}$&\\
\hline
\textbf{1.} &$\eta(0,x),x=ba^\ell y$ &$\color{blue} \lambda_a\eta(0,b^{\textrm{-}1}\alpha(x))$&$\color{blue} \lambda_a\eta(N\textrm{-}1,x)$&$A\overline{B}$&$|A|^2\sqrt{\frac{N\textrm{-}1}{N}}$\\
\emph{9}&$y\in W_{b,b^{\textrm{-}1},e}$&$\color{red} \lambda_a\eta(0,a^{\textrm{-}1}\alpha(x))$&$\color{red} \lambda_a\eta(\ell,y)$&$\overline{A}B$&$\sqrt{\frac{N\textrm{-}\ell}{N}}|B|^2$\\
\emph{14}&$|x|\leq R\textrm{-}1,\ell\geq1$&&&&\\ 
\hline
\textbf{2.}&$\eta(0,x), x=ba^\ell y$ &&$\lambda_a\eta(N\textrm{-}1,x)$&&$|A|^2\sqrt{\frac{N\textrm{-}1}{N}}$\\
\emph{10}&$y\in W_{b,b^{\textrm{-}1},e}$&$\color{red} \lambda_a\eta(0,a^{\textrm{-}1}\alpha(x))$&$\color{red} \lambda_a\eta(\ell,y)$&$\overline{A}B$&$\sqrt{\frac{N\textrm{-}\ell}{N}}|B|^2$\\
&$|x|= R,\ell\geq1$&&&&\\ 
\hline
\textbf{3.}&$\eta(0,x),x=ba^{\ell\textrm{-}N} y$ &$\color{blue} \lambda_a\eta(0,b^{\textrm{-}1}\alpha(x))$&$\color{blue} \lambda_a\eta(N\textrm{-}1,x)$&$A\overline{B}$&$|A|^2\sqrt{\frac{N\textrm{-}1}{N}}$\\
\emph{12}&$y\in W_{b,b^{\textrm{-}1},e}$&$\color{red} \lambda_a\eta(0,a^{\textrm{-}1}\alpha(x))$&$\color{red} \lambda_a\eta(\ell,y)$&$\overline{A}B$&$\sqrt{\frac{N\textrm{-}\ell}{N}}|B|^2$\\
\emph{14}&$|x|\leq R\textrm{-}1,\ell\geq1$&&&&\\ 
\hline
\textbf{4.}&$\eta(0,x),x=ba^{\ell\textrm{-}N} y$ &&$\lambda_a\eta(N\textrm{-}1,x)$&&$|A|^2\sqrt{\frac{N\textrm{-}1}{N}}$\\
\emph{13}&$y\in W_{b,b^{\textrm{-}1},e}$&$\color{red} \lambda_a\eta(0,a^{\textrm{-}1}\alpha(x))$&$\color{red} \lambda_a\eta(\ell,y)$&$\overline{A}B$&$\sqrt{\frac{N\textrm{-}\ell}{N}}|B|^2$\\
&$|x|= R,\ell\geq1$&&&&\\ 
\hline
\textbf{5.}&$\eta(0,x),x=b^\ell y$ &$\color{blue} \lambda_a\eta(0,b^{\textrm{-}1}a^\ell \alpha(y))$&$\color{blue} \lambda_a\eta(N\textrm{-}1,x)$&$A\overline{B}$&$\sqrt{\frac{N\textrm{-}1}{N}}|A|^2$\\
\emph{14}&$y\in W_{a,a^{\textrm{-}1},e}$&$\color{red} \lambda_a\eta(0,b^{\ell\textrm{-}1} y)$&$\color{red} \lambda_a\eta(\ell\textrm{-}1,\alpha(y))$&$|B|^2$&$\sqrt{\frac{N\textrm{-}\ell\textrm{-}1}{N}}\overline{A}B$\\
\emph{15}&$|x|\leq R\textrm{-}1,\ell\geq2$&&&&\\ 
\hline
\textbf{6.}&$\eta(0,x),x=b^\ell y$ &&$\lambda_a\eta(N\textrm{-}1,x)$&&$\sqrt{\frac{N\textrm{-}1}{N}}|A|^2$\\
&$y\in W_{a,a^{\textrm{-}1},e}$&$\color{red} \lambda_a\eta(0,b^{\ell\textrm{-}1} y)$&$\color{red} \lambda_a\eta(\ell\textrm{-}1,\alpha(y))$&$|B|^2$&$\sqrt{\frac{N\textrm{-}\ell\textrm{-}1}{N}}\overline{A}B$\\
&$|x|= R,\ell\geq2$&&&&\\ 
\hline
\textbf{7.}&$\eta(0,x),x=b^{\ell\textrm{-}N} y$ &$\color{blue} \lambda_a\eta(0,b^{\textrm{-}1}a^{\ell\textrm{-}N} \alpha(y))$&$\color{blue} \lambda_a\eta(N\textrm{-}1,x)$&$A\overline{B}$&$\sqrt{\frac{N\textrm{-}1}{N}}|A|^2$\\
\emph{14}&$y\in W_{a,a^{\textrm{-}1},e}$&$\color{blue} \lambda_a\eta(0,b^{\ell\textrm{-}N\textrm{-}1} y)$&$\color{blue} \lambda_a\eta(\ell\textrm{-}1,\alpha(y))$&$|B|^2$&$\sqrt{\frac{\ell\textrm{-}1}{N}}\overline{A}B$\\
\emph{18}&$|x|\leq R\textrm{-}1,\ell\geq2$&&&&\\ 
\hline
\textbf{8.}&$\eta(0,x),x=b^{\ell\textrm{-}N} y$ &&$\lambda_a\eta(N\textrm{-}1,x)$&&$\sqrt{\frac{N\textrm{-}1}{N}}|A|^2$\\
\emph{19}&$y\in W_{a,a^{\textrm{-}1},e}$&&$\lambda_a\eta(\ell\textrm{-}1,\alpha(y))$&&$\sqrt{\frac{\ell\textrm{-}1}{N}}\overline{A}B$\\
&$|x|= R,\ell\geq2$&&&&\\ 
\hline
\textbf{9.}&$\eta(1,x),x=b^\ell y$ &$\color{red} \lambda_a\eta(0,b^\ell y)$&$\color{red} \lambda_a\eta(\ell,\alpha(y))$&$\sqrt{\frac{N\textrm{-}1}{N}}|A|^2$&$\sqrt{\frac{(N\textrm{-}1)(N\textrm{-}\ell)}{N^2}}A\overline{B}$\\
\emph{1}&$y\in W_{a,a^{\textrm{-}1},e}$&$\color{blue}\lambda_a\eta(N\textrm{-}1,b\alpha(x))$&$\color{blue} \lambda_a\eta(0,b^{\textrm{-}1}ax)$&$\sqrt{\frac{N\textrm{-}1}{N}}\overline{A}B$&$\sqrt{\frac{N\textrm{-}1}{N}}|B|^2$\\
\emph{14}&$|x|\leq R\textrm{-}2,\ell\geq1$&&&&\\ 
\hline
\textbf{10.}&$\eta(1,x),x=b^\ell y$ &$\color{red} \lambda_a\eta(0,b^\ell y)$&$\color{red} \lambda_a\eta(\ell,\alpha(y))$&$\sqrt{\frac{N\textrm{-}1}{N}}|A|^2$&$\sqrt{\frac{(N\textrm{-}1)(N\textrm{-}\ell)}{N^2}}A\overline{B}$\\
&$y\in W_{a,a^{\textrm{-}1},e}$&$\lambda_a\eta(N\textrm{-}1,b\alpha(x))$&&$\sqrt{\frac{N\textrm{-}1}{N}}A\overline{B}$&\\
&$|x|= R\textrm{-}1,\ell\geq1$&&&&\\ 
\hline
\end{tabular}
\end{table}

\begin{table}
\caption{} \label{table:table2}
\begin{tabular}{|l|l|ll|ll|}
\hline
\textbf{11.}&$\eta(1,x),x=b^\ell y$ &$\lambda_a\eta(0,b^\ell y)$&$\lambda_a\eta(\ell,\alpha(y))$&$\sqrt{\frac{N\textrm{-}1}{N}}|A|^2$&$\sqrt{\frac{(N\textrm{-}1)(N\textrm{-}\ell)}{N^2}}A\overline{B}$\\
\emph{17}&$y\in W_{a,a^{\textrm{-}1},e}$&&&&\\
&$|x|= R,\ell\geq1$&&&&\\ 
\hline
\textbf{12.}&$\eta(1,x),x=b^{\ell\textrm{-}N} y$ &$\color{red} \lambda_a\eta(0,b^{\ell\textrm{-}N} y)$&$\color{red} \lambda_a\eta(\ell,\alpha(y))$&$\sqrt{\frac{N\textrm{-}1}{N}}|A|^2$&$\sqrt{\frac{(N\textrm{-}1)(\ell)}{N^2}}A\overline{B}$\\
\emph{3}&$y\in W_{a,a^{\textrm{-}1},e}$&$\color{blue} \lambda_a\eta(N\textrm{-}1,b\alpha(x))$&$\color{blue} \lambda_a\eta(0,b^{\textrm{-}1}ax)$&$\sqrt{\frac{N\textrm{-}1}{N}}\overline{A}B$&$\sqrt{\frac{N\textrm{-}1}{N}}|B|^2$\\
\emph{14}&$|x|\leq R\textrm{-}2,\ell\geq1$&&&&\\ 
\hline
\textbf{13.}&$\eta(1,x),x=b^{\ell\textrm{-}N} y$ &$\color{red} \lambda_a\eta(0,b^{\ell\textrm{-}N} y)$&$\color{red} \lambda_a\eta(\ell,\alpha(y))$&$\sqrt{\frac{N\textrm{-}1}{N}}|A|^2$&$\sqrt{\frac{(N\textrm{-}1)(\ell)}{N^2}}A\overline{B}$\\
&$y\in W_{a,a^{\textrm{-}1},e}$&$\lambda_a\eta(N\textrm{-}1,b\alpha(x))$&&$\sqrt{\frac{N\textrm{-}1}{N}}\overline{A}B$&\\
&$|x|= R\textrm{-}1,\ell\geq1$&&&&\\ 
\hline
\textbf{14.}&$\eta(1,x),x=b^{\ell\textrm{-}N} y$ &$\lambda_a\eta(0,b^{\ell\textrm{-}N} y)$&$\lambda_a\eta(\ell,\alpha(y))$&$\sqrt{\frac{N\textrm{-}1}{N}}|A|^2$&$\sqrt{\frac{(N\textrm{-}1)(\ell)}{N^2}}A\overline{B}$\\
*&$y\in W_{a,a^{\textrm{-}1},e}$&&&&\\
&$|x|= R,\ell\geq1$&&&&\\ 
\hline
\textbf{15.}&$\eta(k,x)$&$\color{red} \lambda_a\eta(k\textrm{-}1,x)$&$\color{red} \lambda_a\eta(0,b^{k\textrm{-}1}\alpha(x))$&$s(k,N)$&$\sqrt{\frac{N\textrm{-}k}{N}}A\overline{B}$\\
\emph{5}&$2\leq k\leq N\textrm{-}1$&$\color{blue} \lambda_a\eta(0,b^{\textrm{-}1}a^kx)$&$\color{blue} \lambda_a\eta(N\textrm{-}1,b^k\alpha(x))$&$\sqrt{\frac{N\textrm{-}k}{N}}|B|^2$&$\sqrt{\frac{(N\textrm{-}1)(N\textrm{-}k)}{N^2}}\overline{A}B$\\
\emph{14}&$k+|x|\leq R\textrm{-}1$&&&&\\
\hline
\textbf{16.}&$\eta(k,x)$&$\color{red} \lambda_a\eta(k\textrm{-}1,x)$&$\color{red} \lambda_a\eta(0,b^{k\textrm{-}1}\alpha(x))$&$s(k,N)$&$\sqrt{\frac{N\textrm{-}k}{N}}A\overline{B}$\\
&$2\leq k\leq N\textrm{-}1$&&$\lambda_a\eta(N\textrm{-}1,b^k\alpha(x))$&&$\sqrt{\frac{(N\textrm{-}1)(N\textrm{-}k)}{N^2}}\overline{A}B$\\
&$k+|x|= R$&&&&\\
\hline
\textbf{17.}&$\eta(k,x)$&$\lambda_a\eta(k\textrm{-}1,x)$&$\lambda_a\eta(0,b^{k\textrm{-}1}\alpha(x))$&$s(k,N)$&$\sqrt{\frac{N\textrm{-}k}{N}}A\overline{B}$\\
\emph{11}&$2\leq k\leq N\textrm{-}1$&&&&\\
&$k+|x|= R+1$&&&&\\
\hline
\textbf{18.}&$\eta(k,x)$&$\color{blue} \lambda_a\eta(k\textrm{-}1,x)$&$\color{blue} \lambda_a\eta(0,b^{k\textrm{-}N\textrm{-}1}\alpha(x))$&$s(k,N)$&$\sqrt{\frac{k}{N}}A\overline{B}$\\
\emph{7}&$2\leq k\leq N\textrm{-}1$&$\color{blue} \lambda_a\eta(0,b^{\textrm{-}1}a^{k\textrm{-}N}x)$&$\color{blue} \lambda_a\eta(N\textrm{-}1,b^{k\textrm{-}N}\alpha(x))$&$\sqrt{\frac{k}{N}}|B|^2$&$\sqrt{\frac{(N\textrm{-}1)k}{N^2}}\overline{A}B$\\
\emph{14}&$N\textrm{-}k+|x|\leq R\textrm{-}1$&&&&\\
\hline
\textbf{19.}&$\eta(k,x)$&$\lambda_a\eta(k\textrm{-}1,x)$&&$s(k,N)$&\\
\emph{8}&$2\leq k\leq N\textrm{-}1$&&$\lambda_a\eta(N\textrm{-}1,b^{k\textrm{-}N}\alpha(x))$&&$\sqrt{\frac{(N\textrm{-}1)k}{N^2}}\overline{A}B$\\
&$N\textrm{-}k+|x|= R$&&&&\\
\hline
\textbf{20.}&$\eta(k,x),2\leq k$&$\lambda_a\eta(k\textrm{-}1,x)$&&$s(k,N)$&\\
&$R+1\textrm{-}|x|<k$&&&&\\
&$k<N\textrm{-}R+|x|$&&&&\\
\hline
\end{tabular}
\end{table}

\begin{table} 
\caption{} \label{table:table3}
\begin{tabular}{|l|l|ll|ll|}
\hline
&$\lambda_a\eta(0,e)$& $\eta(1,e)$ & $\eta(0,b)$  &$\sqrt{\frac{N\textrm{-}1}{N}}\overline{A}$&$\overline{B}$\\
\hline
&$\lambda_a\eta(0,b^{\textrm{-}1})$& $\eta(1,b^{\textrm{-}1})$ & $\eta(0,e)$  &$\sqrt{\frac{N\textrm{-}1}{N}}|A|^2$&$B$\\
&&$\eta(0,ba^{\textrm{-}1})$&&$B\overline{A}$&\\
\hline
\textbf{1.}&$\lambda_a\eta(0,x),x=b^{\textrm{-}1}a^{\textrm{-}\ell} y$ &$\eta(0,b\alpha(x))$&$\eta(1,x)$&$A\overline{B}$&$|A|^2\sqrt{\frac{N\textrm{-}1}{N}}$\\
&$y\in W_{b,b^{\textrm{-}1},e}$&$\eta(0,a\alpha(x))$&$\eta(N\textrm{-}\ell,y)$&$\overline{A}B$&$\sqrt{\frac{N\textrm{-}\ell}{N}}|B|^2$\\
&$|x|\leq R\textrm{-}1,\ell\geq1$&&&&\\ 
\hline
\textbf{2.}&$\lambda_a\eta(0,x),x=b^{\textrm{-}1}a^{\textrm{-}\ell} y$ &&$\eta(1,x)$&&$|A|^2\sqrt{\frac{N\textrm{-}1}{N}}$\\
&$y\in W_{b,b^{\textrm{-}1},e}$&$\eta(0,a\alpha(x))$&$\eta(N\textrm{-}\ell,y)$&$\overline{A}B$&$\sqrt{\frac{N\textrm{-}\ell}{N}}|B|^2$\\
&$|x|= R,\ell\geq1$&&&&\\ 
\hline
\textbf{3.}&$\lambda_a\eta(0,x),x=b^{\textrm{-}1}a^{N\textrm{-}\ell} y$ &$\eta(0,b\alpha(x))$&$\eta(1,x)$&$A\overline{B}$&$|A|^2\sqrt{\frac{N\textrm{-}1}{N}}$\\
&$y\in W_{b,b^{\textrm{-}1},e}$&$\eta(0,a\alpha(x))$&$\eta(N\textrm{-}\ell,y)$&$\overline{A}B$&$\sqrt{\frac{N\textrm{-}\ell}{N}}|B|^2$\\
&$|x|\leq R$\textrm{-}$1,\ell\geq1$&&&&\\ 
\hline
\textbf{4.}&$\lambda_a\eta(0,x),x=b^{\textrm{-}1}a^{N\textrm{-}\ell} y$ &&$\eta(1,x)$&&$|A|^2\sqrt{\frac{N\textrm{-}1}{N}}$\\
&$y\in W_{b,b^{\textrm{-}1},e}$&$\eta(0,a\alpha(x))$&$\eta(N\textrm{-}\ell,y)$&$\overline{A}B$&$\sqrt{\frac{N\textrm{-}\ell}{N}}|B|^2$\\
&$|x|= R,\ell\geq1$&&&&\\ 
\hline
\textbf{5.}&$\lambda_a\eta(0,x),x=b^{\textrm{-}\ell} y$ &$\eta(0,b\alpha(x))$&$\eta(1,x)$&$A\overline{B}$&$\sqrt{\frac{N\textrm{-}1}{N}}|A|^2$\\
&$y\in W_{a,a^{\textrm{-}1},e}$&$\eta(0,bx)$&$\eta(N\textrm{-}\ell+1,\alpha(y))$&$|B|^2$&$\sqrt{\frac{N\textrm{-}\ell\textrm{-}1}{N}}\overline{A}B$\\
&$|x|\leq R\textrm{-}1,\ell\geq2$&&&&\\ 
\hline
\textbf{6.}&$\lambda_a\eta(0,x),x=b^{\textrm{-}\ell} y$ &&$\eta(1,x)$&&$\sqrt{\frac{N\textrm{-}1}{N}}|A|^2$\\
&$y\in W_{a,a^{\textrm{-}1},e}$&$\eta(0,bx)$&$\eta(N\textrm{-}\ell+1,\alpha(y))$&$|B|^2$&$\sqrt{\frac{N\textrm{-}\ell\textrm{-}1}{N}}\overline{A}B$\\
&$|x|= R,\ell\geq2$&&&&\\ 
\hline
\textbf{7.}&$\lambda_a\eta(0,x),x=b^{N\textrm{-}\ell} y$ &$\eta(0,b\alpha(x))$&$\eta(1,x)$&$A\overline{B}$&$\sqrt{\frac{N\textrm{-}1}{N}}|A|^2$\\
&$y\in W_{a,a^{\textrm{-}1},e}$&$\eta(0,bx)$&$\eta(N\textrm{-}\ell+1,\alpha(y))$&$|B|^2$&$\sqrt{\frac{\ell\textrm{-}1}{N}}\overline{A}B$\\
&$|x|\leq R\textrm{-}1,\ell\geq2$&&&&\\ 
\hline
\textbf{8.}&$\lambda_a\eta(0,x),x=b^{N\textrm{-}\ell} y$ &&$\eta(1,x)$&&$\sqrt{\frac{N\textrm{-}1}{N}}|A|^2$\\
&$y\in W_{a,a^{\textrm{-}1},e}$&&$\eta(N\textrm{-}\ell+1,\alpha(y))$&&$\sqrt{\frac{\ell\textrm{-}1}{N}}\overline{A}B$\\
&$|x|= R,\ell\geq2$&&&&\\ 
\hline
\textbf{9.}&$\lambda_a\eta(N\textrm{-}1,x),x=b^{\textrm{-}\ell} y$ &$\eta(0,x)$&$\eta(N\textrm{-}\ell,\alpha(y))$&$\sqrt{\frac{N\textrm{-}1}{N}}|A|^2$&$\sqrt{\frac{(N\textrm{-}1)(N\textrm{-}\ell)}{N^2}}A\overline{B}$\\
&$y\in W_{a,a^{\textrm{-}1},e}$&$\eta(1,b^{\textrm{-}1}\alpha(x))$&$\eta(0,ba^{\textrm{-}1}x)$&$\sqrt{\frac{N\textrm{-}1}{N}}\overline{A}B$&$\sqrt{\frac{N\textrm{-}1}{N}}|B|^2$\\
&$|x|\leq R\textrm{-}2,\ell\geq1$&&&&\\ 
\hline
\textbf{10.}&$\lambda_a\eta(N\textrm{-}1,x),x=b^{\textrm{-}\ell} y$ &$\eta(0,x)$&$\eta(N\textrm{-}\ell,\alpha(y))$&$\sqrt{\frac{N\textrm{-}1}{N}}|A|^2$&$\sqrt{\frac{(N\textrm{-}1)(N\textrm{-}\ell)}{N^2}}A\overline{B}$\\
&$y\in W_{a,a^{\textrm{-}1},e}$&$\eta(1,b^{\textrm{-}1}\alpha(x))$&&$\sqrt{\frac{N\textrm{-}1}{N}}A\overline{B}$&\\
&$|x|= R\textrm{-}1,\ell\geq1$&&&&\\ 
\hline
\end{tabular}
\end{table}

\begin{table}
\caption{} \label{table:table4}
\begin{tabular}{|l|l|ll|ll|} 
\hline
\textbf{11.}&$\lambda_a\eta(N\textrm{-}1,x),x=b^{\textrm{-}\ell} y$ &$\eta(0,x)$&$\eta(N\textrm{-}\ell,\alpha(y))$&$\sqrt{\frac{N\textrm{-}1}{N}}|A|^2$&$\sqrt{\frac{(N\textrm{-}1)(N\textrm{-}\ell)}{N^2}}A\overline{B}$\\
&$y\in W_a\cup W_{a^{\textrm{-}1}}\cup\{ e \}$&&&&\\
&$|x|= R,\ell\geq1$&&&&\\ 
\hline
\textbf{12.}&$\lambda_a\eta(N\textrm{-}1,x),x=b^{N\textrm{-}\ell} y$ &$\eta(0,x)$&$\eta(N\textrm{-}\ell,\alpha(y))$&$\sqrt{\frac{N\textrm{-}1}{N}}|A|^2$&$\sqrt{\frac{(N\textrm{-}1)(\ell)}{N^2}}A\overline{B}$\\
&$y\in W_a\cup W_{a^{\textrm{-}1}}\cup\{ e \}$&$\eta(1,b^{\textrm{-}1}\alpha(x))$&$\eta(0,ba^{\textrm{-}1}x)$&$\sqrt{\frac{N\textrm{-}1}{N}}\overline{A}B$&$\sqrt{\frac{N\textrm{-}1}{N}}|B|^2$\\
&$|x|\leq R\textrm{-}2,\ell\geq1$&&&&\\ 
\hline
\textbf{13.}&$\lambda_a\eta(N\textrm{-}1,x),x=b^{N\textrm{-}\ell} y$ &$\eta(0,x)$&$\eta(N\textrm{-}\ell,\alpha(y))$&$\sqrt{\frac{N\textrm{-}1}{N}}|A|^2$&$\sqrt{\frac{(N\textrm{-}1)(\ell)}{N^2}}A\overline{B}$\\
&$y\in W_a\cup W_{a^{\textrm{-}1}}\cup\{ e \}$&$\eta(1,b^{\textrm{-}1}\alpha(x))$&&$\sqrt{\frac{N\textrm{-}1}{N}}\overline{A}B$&\\
&$|x|= R\textrm{-}1,\ell\geq1$&&&&\\ 
\hline
\textbf{14.}&$\lambda_a\eta(N\textrm{-}1,x),x=b^{N\textrm{-}\ell} y$ &$\eta(0,x)$&$\eta(N\textrm{-}\ell,\alpha(y))$&$\sqrt{\frac{N\textrm{-}1}{N}}|A|^2$&$\sqrt{\frac{(N\textrm{-}1)(\ell)}{N^2}}A\overline{B}$\\
&$y\in W_a\cup W_{a^{\textrm{-}1}}\cup\{ e \}$&&&&\\
&$|x|= R,\ell\geq1$&&&&\\ 
\hline
\textbf{15.}&$\lambda_a\eta(N\textrm{-}k,x)$&$\eta(N\textrm{-}k+1,x)$&$\eta(0,b^{1\textrm{-}k}\alpha(x))$&$s(k,N)$&$\sqrt{\frac{N\textrm{-}k}{N}}A\overline{B}$\\
&$2\leq k\leq N\textrm{-}1$&$\eta(0,ba^{\textrm{-}k}x)$&$\eta(1,b^{\textrm{-}k}\alpha(x))$&$\sqrt{\frac{N\textrm{-}k}{N}}|B|^2$&$\sqrt{\frac{(N\textrm{-}1)(N\textrm{-}k)}{N^2}}\overline{A}B$\\
&$k+|x|\leq R\textrm{-}1$&&&&\\
\hline
\textbf{16.}&$\lambda_a\eta(N\textrm{-}k,x)$&$\eta(N\textrm{-}k+1,x)$&$\eta(0,b^{1\textrm{-}k}\alpha(x))$&$s(k,N)$&$\sqrt{\frac{N\textrm{-}k}{N}}A\overline{B}$\\
&$2\leq k\leq N\textrm{-}1$&&$\eta(1,b^{\textrm{-}k}\alpha(x))$&&$\sqrt{\frac{(N\textrm{-}1)(N\textrm{-}k)}{N^2}}\overline{A}B$\\
&$k+|x|= R$&&&&\\
\hline
\textbf{17.}&$\lambda_a\eta(N\textrm{-}k,x)$&$\eta(N\textrm{-}k+1,x)$&$\eta(0,b^{1\textrm{-}k}\alpha(x))$&$s(k,N)$&$\sqrt{\frac{N\textrm{-}k}{N}}A\overline{B}$\\
&$2\leq k\leq N\textrm{-}1$&&&&\\
&$k+|x|= R+1$&&&&\\
\hline
\textbf{18.}&$\lambda_a\eta(N\textrm{-}k,x)$&$\eta(N\textrm{-}k+1,x)$&$\eta(0,b^{N\textrm{-}k+1}\alpha(x))$&$s(k,N)$&$\sqrt{\frac{k}{N}}A\overline{B}$\\
&$2\leq k\leq N\textrm{-}1$&$\eta(0,ba^{N\textrm{-}k}x)$&$\eta(1,b^{N\textrm{-}k}\alpha(x))$&$\sqrt{\frac{k}{N}}|B|^2$&$\sqrt{\frac{(N\textrm{-}1)k}{N^2}}\overline{A}B$\\
&$N\textrm{-}k+|x|\leq R\textrm{-}1$&&&&\\
\hline
\textbf{19.}&$\lambda_a\eta(N\textrm{-}k,x)$&$\eta(N\textrm{-}k+1,x)$&&$s(k,N)$&\\
&$2\leq k\leq N\textrm{-}1$&&$\eta(1,b^{N\textrm{-}k}\alpha(x))$&&$\sqrt{\frac{(N\textrm{-}1)k}{N^2}}\overline{A}B$\\
&$N\textrm{-}k+|x|= R$&&&&\\
\hline
\textbf{20.}&$\lambda_a\eta(N\textrm{-}k,x),k\geq 2$&$\eta(N\textrm{-}k+1,x)$&&$s(k,N)$&\\
&$R+1\textrm{-}|x|<k$&&&&\\
&$k<N\textrm{-}R+|x|$&&&&\\
\hline
\end{tabular}
\end{table}

\newpage
\subsection{Proof of Lemma \ref{lem:coloredtable}} \label{appsec:lemmaproof}
\begin{proof}[Proof of Lemma \ref{lem:coloredtable}] We first use Tables \ref{table:table1} and \ref{table:table2} to build Tables \ref{table:table3} and \ref{table:table4}. Consider the automorphism $\beta$ of $\F_2$ defined by $\beta(a)=a^{-1}$ and $\beta(b)=b^{-1}.$ Clearly $\alpha$ (the order two automorphism swapping $a$ and $b$) and $\beta$ commute.  Let $W\in B(\ell^2\F_2)$ be the unitary defined by $W(\delta_x)=\delta_{\beta(x)}.$
We have $W(\eta(0,e))=\eta(0,e)$ and when $x\neq e$
\begin{equation*}
W\eta(0,x)=W(A\delta_x+B\delta_{\alpha(x)})=A\delta_{\beta(x)}+B\delta_{\alpha\beta(x)}=\eta(0,\beta(x))
\end{equation*}
When $k\neq0$ we have
\begin{align*}
W\eta(k,x)&=\sqrt{\frac{k}{N}}(A\delta_{a^{N-k}\beta(x)}+B\delta_{b^{N-k}\alpha\beta(x)})+\sqrt{\frac{N-k}{N}}(A\delta_{a^{-k}\beta(x)}+B\delta_{b^{-k}\alpha\beta(x)})\\
&=\sqrt{\frac{N-k}{N}}(A\delta_{a^{(N-k)-N}\beta(x)}+B\delta_{b^{-k}\alpha\beta(x)})\\
&+\sqrt{\frac{N-(N-k)}{N}}(A\delta_{a^{N-k}\beta(x)}+B\delta_{b^{N-k}\alpha\beta(x)})\\
&=\eta(N-k,\beta(x))
\end{align*}
We therefore have
\begin{equation*}
\la \eta(k,x),\lambda_{a^{-1}}\eta(j,y)  \ra=\la  \eta(N-k,\beta(x)),\lambda_a\eta(N-j,\beta(y))  \ra
\end{equation*}
We can therefore derive tables \ref{table:table3} and \ref{table:table4} below from tables \ref{table:table1} and \ref{table:table2} above. Notice that in tables \ref{table:table3} and \ref{table:table4} we moved the $\lambda_{a^{-1}}$ over to the first column to improve readability. 
\\\\
Now suppose that $(k,x)\in S.$  It follows easily (but not effortlessly) from Tables \ref{table:table3} and \ref{table:table4} that $\la Q\lambda_a\eta(k,x),\lambda_a\eta(k,x)\ra\geq \frac{N-R}{N}.$ Let us show a representative example.  Take for example a vector $\lambda_a\eta(0,x)$ with $x=b^{-1}a^{-\ell} y$ from box \textbf{1} in Table \ref{table:table3}. Then
\begin{align*}
&\la Q \lambda_a\eta(0,x),\lambda_a\eta(0,x)\ra\\
&=|\la\lambda_a\eta(0,x),\eta(0,b\alpha(x))\ra|^2+|\la\lambda_a\eta(0,x),\eta(1,x)\ra|^2\\
&+|\la\lambda_a\eta(0,x),\eta(0,a\alpha(x))\ra|^2+|\la\lambda_a\eta(0,x),\eta(N-\ell,y)\ra|^2\\
&=\frac{1}{4}\Big(1+\frac{N-1}{N}+1+\frac{N-\ell}{N}\Big)\\
&\geq \frac{N-R}{N}.
\end{align*}
Finally for any $\xi\in \textup{span}\{\lambda_a\eta(k,x):(k,x)\in S   \}$ of norm 1, basic estimates show that
\begin{equation*}
\la  Q\xi,\xi \ra\geq 1-\frac{R}{N}-\sqrt{\frac{R}{N}}|S|\geq 1-\frac{R}{N}-\sqrt{\frac{R}{N}}12\cdot 3^{R-1}\geq 1-\frac{1}{N^{1/4}}.
\end{equation*}
The last estimate follows from the relative value of $R$ to $N$ from Definition \ref{def:NRdef}.
\end{proof}

\newpage
\bibliographystyle{plain}

\begin{thebibliography}{10}

\bibitem{Brown06}
Nathanial~P. Brown.
\newblock Invariant means and finite representation theory of {$C^*$}-algebras.
\newblock {\em Mem. Amer. Math. Soc.}, 184(865), 2006.

\bibitem{Brown13}
Nathanial~P. Brown and Erik~P. Guentner.
\newblock New {$\rm C^\ast$}-completions of discrete groups and related spaces.
\newblock {\em Bull. Lond. Math. Soc.}, 45(6):1181--1193, 2013.

\bibitem{Brown08}
Nathanial~P. Brown and Narutaka Ozawa.
\newblock {\em {$C^*$}-algebras and finite-dimensional approximations},
  volume~88 of {\em Graduate Studies in Mathematics}.
\newblock American Mathematical Society, Providence, RI, 2008.

\bibitem{Carrion13}
Jos{\'e}~R. Carri{\'o}n, Marius Dadarlat, and Caleb Eckhardt.
\newblock On groups with quasidiagonal {$C^*$}-algebras.
\newblock {\em J. Funct. Anal.}, 265(1):135--152, 2013.

\bibitem{Choi79}
Man~Duen Choi.
\newblock A simple {$C^{\ast} $}-algebra generated by two finite-order
  unitaries.
\newblock {\em Canad. J. Math.}, 31(4):867--880, 1979.

\bibitem{Choi80}
Man~Duen Choi.
\newblock The full {$C\sp{\ast} $}-algebra of the free group on two generators.
\newblock {\em Pacific J. Math.}, 87(1):41--48, 1980.

\bibitem{Effros00}
Edward~G. Effros and Zhong-Jin Ruan.
\newblock {\em Operator spaces}, volume~23 of {\em London Mathematical Society
  Monographs. New Series}.
\newblock The Clarendon Press Oxford University Press, New York, 2000.

\bibitem{Haagerup79}
Uffe Haagerup.
\newblock An example of a nonnuclear {$C^{\ast} $}-algebra, which has the
  metric approximation property.
\newblock {\em Invent. Math.}, 50(3):279--293, 1978/79.

\bibitem{Hadwin87}
Don Hadwin.
\newblock Strongly quasidiagonal {$C^*$}-algebras.
\newblock {\em J. Operator Theory}, 18(1):3--18, 1987.
\newblock With an appendix by Jonathan Rosenberg.

\bibitem{Okayasu14}
Rui Okayasu.
\newblock Free group {$C^*$}-algebras associated with {$\ell_p$}.
\newblock {\em Internat. J. Math.}, 25(7):1450065, 12, 2014.

\bibitem{Pimsner79}
M.~Pimsner, S.~Popa, and D.~Voiculescu.
\newblock Homogeneous {$C^{\ast} $}-extensions of {$C(X)\otimes K(H)$}. {I}.
\newblock {\em J. Operator Theory}, 1(1):55--108, 1979.

\bibitem{Pisier98}
Gilles Pisier.
\newblock Non-commutative vector valued {$L_p$}-spaces and completely
  {$p$}-summing maps.
\newblock {\em Ast\'erisque}, (247), 1998.

\bibitem{Pisier03}
Gilles Pisier.
\newblock {\em Introduction to operator space theory.}
\newblock Number 294 in London Mathematical Society Lecture Note Series.
  Cambridge University Press, Cambridge, 2003.

\bibitem{Pytlik86}
T.~Pytlik and R.~Szwarc.
\newblock An analytic family of uniformly bounded representations of free
  groups.
\newblock {\em Acta Math.}, 157(3-4):287--309, 1986.

\bibitem{Ruan16}
Zhong-Jin Ruan and Matthew Wiersma.
\newblock On exotic group {C}*-algebras.
\newblock {\em J. Funct. Anal.}, 271(2):437--453, 2016.

\bibitem{Voiculescu76}
Dan Voiculescu.
\newblock A non-commutative {W}eyl-von {N}eumann theorem.
\newblock {\em Rev. Roumaine Math. Pures Appl.}, 21(1):97--113, 1976.

\end{thebibliography}

\end{document}